\documentclass[a4paper,11pt,twoside]{article}
\usepackage{eurosym}
\usepackage{geometry}
\usepackage[utf8]{inputenc}
\usepackage{lmodern}
\usepackage[T1]{fontenc}
\usepackage{dsfont}
\usepackage{amsmath,amssymb,amsthm,empheq,cases}
\usepackage{hyperref}
\usepackage{tikz}
\usepackage{graphicx}
\usepackage{url}

\hypersetup{colorlinks,
            citecolor=red, 
            filecolor=black,
            linkcolor=blue,
            urlcolor=black}
\def \dis {\displaystyle}
\def \NN {\mathbb N}

\def \RR {\mathbb R}
\def \CC {\mathbb C}

\def \B {\mathcal{B}}

\def \D {\mathcal{D}}
\def \E {\mathcal{E}}

\def \L {\mathcal{L}}

\def \ecart {\noalign{\medskip}}
\theoremstyle{definition}
\newtheorem{Th}{Theorem}[section]
\newtheorem{Prop}[Th]{Proposition}
\newtheorem{Lem}[Th]{Lemma}

\newtheorem{Def}[Th]{Definition}
\newtheorem{Rem}[Th]{Remark}

\def \refD #1{Definition~\ref{#1}}
\def \refT #1{Theorem~\ref{#1}}
\def \refL #1{Lemma~\ref{#1}}

\def \refP #1{Proposition~\ref{#1}}

\title{Analytic semigroup generated by the dispersal process of a sylvatic transmission model of Chagas disease}

\author{Narimene Benarbia, Rabah Labbas, Tewfik Mahdjoub, Alexandre Thorel \medskip \\
{\scriptsize N. B. \& T. M.: NLAAML, Abou-Bekr Belka\"{i}d University, Tlemcen, 13000, Algeria} \\
{\scriptsize narimene.benarbia@univ-tlemcen.dz, tewfik.mahdjoub@univ-tlemcen.dz} \\
{\scriptsize R. L. \& A. T.: Université Le Havre Normandie, Normandie Univ, LMAH UR 3821, 76600 Le Havre, France.} \\
{\scriptsize rabah.labbas@univ-lehavre.fr, alexandre.thorel@univ-lehavre.fr}}

\date{\empty}

\begin{document}

\maketitle\vspace{-0.5cm}

\begin{abstract}
In this work, we develop a new biological transmission model for Chagas disease. This model, set in two juxtaposed habitats with skew Brownian motion conditions at the interface, is composed of two reaction-diffusion equations and takes into account the sylvatic transmission. We write it as an abstract perturbed Cauchy problem using operator theory. Then, we show that the main operator, which models the dispersal process, generates an analytic semigroup in an adequate Banach space. 
\end{abstract} 

\section{Introduction}

According to World Health Organization (WHO) statistics, between six and seven million people of the world population are affected by Chagas disease see \cite{telleria2017} and \cite{who2023}. The disease, endemic in Latin
America with a morbidity of 12 000 deaths/year, has seen its geographical
expansion reach the countries of North America, Europe, Australia and Asia 
\cite{paho2023}. On the African continent, the disease is present, for instance, in Gabon and Tanzania with less than 900 cases/country.

Chagas disease or American trypanosomiasis is a vector-borne disease caused
by the flagellate protozoan parasite \textit{Trypanosoma cruzi} (\textit{T.
cruzi}). This parasite can infect a wide range of mammalian hosts (including
humans) or domestic or wild birds. The main mode of transmission is contact
with insect vectors called triatominae. When the latter live in shelters
close to human dwellings, the transmission of the parasite is said to be
domestic. If they live in the nests of mammals or undomesticated birds, the
transmission is sylvatic. Some vector species such as \textit{T. Dimidiata} have
adapted to both domestic and sylvatic habitats: transmission in this case is
peridomestic.

During the typical infection cycle, called the kissing bug, an infected
triatomine takes a host's blood meal and releases the parasite in their
feces near the bite. The parasite enters the host through the wound or an
intact mucous membrane, such as the conjunctiva: this is called sterocorary
transmission. An infected host also transmits the parasite to a healthy
vector during blood meals and the parasite then resides in the intestine of
the vector. Transmission in hosts can also occur congenitally, (vertical
transmission from infected mother to child) and orally.

Vector control remained the main strategy to counter the spread of the
disease. It was based on the regular spraying of insecticide in the affected
villages. This fight, carried out jointly by several Latin American
countries within the framework of intergovernmental projects, cost US\$ 30
million annually without being able to achieve the objectives set by the
WHO. Indeed, areas where the presence of the vector has been drastically
reduced, have seen the installation of a process of re-infestation by
triatomines \cite{arias2023}.

The objective of mathematical modeling is therefore to try to explain this
phenomenon and then to determine the main demographic parameters that make a
triatomine invasion successful in order to allow decision-makers to control these parameters. 
It was based on various mathematical tools selected according to the established objectives. To demonstrate that, in the sylvatic (wild) case, oral contamination is as significant as classical transmission via bites during a blood meal, \cite{kribs2006} proposed an SI model based on ordinary differential equations (ODE). This model incorporates a prey-predator sub-model where the vector seeks a host for a blood meal, while the host -generally insectivorous- feeds on the vector.
Model analysis shows that if the vector-host contact rate rapidly reaches saturation (the host reaches satiety quickly), then the vector population exhibits two stable equilibrium states: one at high vector density and the other at low vector density. Controlling the infection by shifting it from high to low density involves crossing a critical threshold. If this threshold is not reached, the population returns to its equilibrium state. The drawback of this model is that it does not take the spatial component into account.\\
The geographical habitat of vectors was considered in the model proposed by \cite{mesk2016}. This work follows a "population ecology" approach and places the vector population in a scenario of invading a pristine sylvatic space. This biological invasion process is modeled using traveling waves, which are solutions to a system of integro-difference equations. Using this tool, the invasion speed representing the displacement velocity of traveling waves moving from colonized to non-colonized space is estimated for the vector population. However, the model cannot conclude on the invasion speed of \textit{T. cruzi}, which is not necessarily the same as its vector's invasion speed. Furthermore, this model assumes that once a habitat point is colonized, it cannot be decolonized. This implies a certain advection of the vector flux, which is not observed in the field.\\
Still aiming to estimate the spread speed of the \textit{T. cruzi} parasite in the sylvatic case, practical studies have been developed. A first cellular automaton (CA) model, constructed by \cite{crawford2013}, divides the geographic area between Mexico and the southern United States into a grid of 9,376 contiguous cells. In each cell, two systems of differential equations are written, integrating the contacts of two vector/host pairs. These ODE systems, solved numerically, allow for the deduction of \textit{T. cruzi} prevalence in each cell. On the grid, prevalences exceeding 7\% define infected cells, enabling the deduction of the disease's invasion speed. Although Kribs' CA provides invasion speeds depending on the relief of the patch, the model does not explicitly inform decision-makers on which factor to act -vector or hosts- to stop the infection.\\
To answer this latter question, \cite{elsaadi2020} developed an agent-based model (multi-agent model) accounting for more complex interactions between hosts and vectors. Several infection inoculation scenarios on a 100x100 cell grid were monitored, and numerical simulations show that the infection is initially carried by vectors and then persists thanks to the hosts. Consequently, disease control should not be limited to insecticide use but must involve action on the hosts, especially since the model shows that vertical transmission has a significant impact on maintaining high prevalence.\\ 
 In the context of peridomestic transmission, a model based on reaction-diffusion equations was proposed by \cite{dib2021}. This type of partial differential equation (PDE) integrates both spatial dispersion processes and the demographics of vectors and hosts. 
 The model defines a buffer zone between the village and the source of infection (the forest), where the \textit{T. cruzi} parasite is present following Brownian motion. Each habitat point is not permanently colonized, which more closely reflects field reality. The authors demonstrate that the problem is mathematically well-posed, providing the opportunity for subsequent numerical exploitation.\\
In the present work, we maintain the principle of the existence of Brownian motion of the parasite on either side of the boundary between the infected zone and the healthy zone in the case of sylvatic transmission. Vertical transmission of the parasite in hosts is also integrated into the reaction-diffusion equations.
\\
Sylvatic transmission is considered over a domain $\Omega =\left[-\ell,L\right]\times\left[ 0,1\right] \subset \mathbb{R}^2$, representing the natural habitat of the vectors and the hosts, considered in two health states: susceptible and infected. The susceptible individuals live on $\Omega _{S}=\left[ 0,L\right]\times\left[ 0,1\right]$ while the infected one are on $\Omega _{I}=\left[ -\ell,0\right] \times\left[0,1\right]$. Thus, the spatio-temporal dynamics of vector and host populations induce skew Brownian motion conditions of individuals across the common border $\Gamma =\left\{0\right\} \times\left[ 0,1\right]$, see equations \eqref{asym} below. On the other hand, the demographic and spatial dispersal processes, on the juxtaposed subdomains $\Omega_{S}$ and $\Omega _{I}$ are described by reaction-diffusion equations, see \eqref{Syst Reaction-diffusion} below. The system obtained is then written, in an adequate functional space, in the form of an abstract differential equation. \\
Then, our aim is to show that the principal dispersal operator generates an analytic semigroup.

\section{The model}

\subsection{Demography and disease transmission}

The life cycle of the Triatominae consists of seven stages of development:
an egg stage, five larval stages and an adult stage. Following \cite{menu2010}, the development from the egg to the fifth larval stage is considered as a single stage called the juvenile stage. Let denote respectively by $J(t,x,y)$, $A(t,x,y)$ and $H(t,x,y)$, the densities of juvenile vectors, adult vectors and hosts, at time $t$ at point $(x,y)$. Two health states are considered: susceptible individuals, non-carriers of the \textit{T.Cruzi} parasite, of densities $J_{S}$, $A_{S}$ and $H_{S}$ and infected individuals, carriers of the parasite, of densities $J_{I}$, $A_{I}$
and $H_{I}$.
\\
Between $t$ and $t+dt$, the juveniles having survived until $t$ with a
probability $\sigma _{J}$ will pass to the adult stage with a probability $\tau _{J}$ or will remain juveniles with a probability $(1-\tau _{J})$. Adults that survived with probability $\sigma _{A}$ will lay eggs with fecundity rate $f_{A}$. The host mammals have a survival rate $\sigma _{H}$ and a fecundity $f_{H}$. For simplicity, we will consider that all these demographic parameters are constant. Their interval are given in Table \ref{tab:my_label}. The life cycles of vectors and hosts are schematized in Figure \ref{fig:my_label} below.

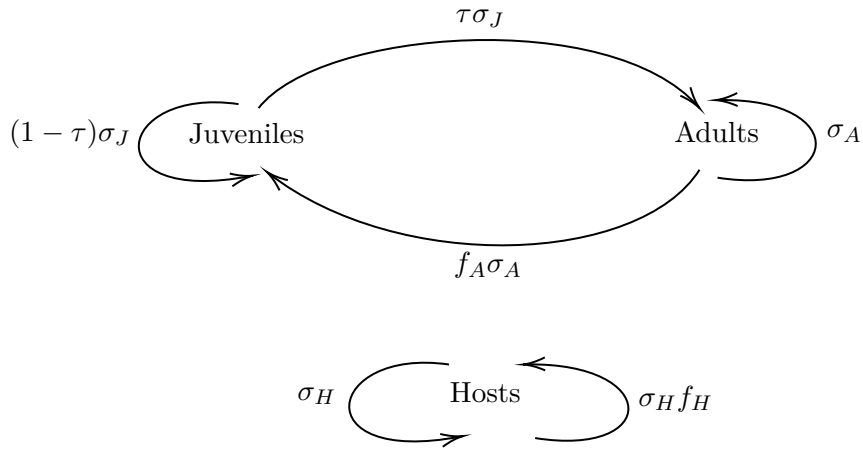
\begin{figure}[ht]
\begin{center}
\tikzset{every picture/.style={line width=0.75pt}} %set default line width to 0.75pt        

\begin{tikzpicture}[x=0.75pt,y=0.75pt,yscale=-1,xscale=1]
%uncomment if require: \path (0,400); %set diagram left start at 0, and has height of 400

%Curve Lines [id:da155187887068567] 
\draw    (188,121) .. controls (218.85,81.2) and (363.54,70.11) .. (407.35,120.24) ;
\draw [shift={(408,121)}, rotate = 229.86] [color={rgb, 255:red, 0; green, 0; blue, 0 }  ][line width=0.75]    (10.93,-3.29) .. controls (6.95,-1.4) and (3.31,-0.3) .. (0,0) .. controls (3.31,0.3) and (6.95,1.4) .. (10.93,3.29)   ;
%Curve Lines [id:da18172205256906215] 
\draw    (408,152) .. controls (378.15,199.76) and (258.21,204.95) .. (192.98,153.78) ;
\draw [shift={(192,153)}, rotate = 38.66] [color={rgb, 255:red, 0; green, 0; blue, 0 }  ][line width=0.75]    (10.93,-3.29) .. controls (6.95,-1.4) and (3.31,-0.3) .. (0,0) .. controls (3.31,0.3) and (6.95,1.4) .. (10.93,3.29)   ;
%Curve Lines [id:da876864303494783] 
\draw    (418.05,116.96) .. controls (484.89,116.01) and (479.37,165.9) .. (417,156) ;
\draw [shift={(416,117)}, rotate = 358.34] [color={rgb, 255:red, 0; green, 0; blue, 0 }  ][line width=0.75]    (10.93,-3.29) .. controls (6.95,-1.4) and (3.31,-0.3) .. (0,0) .. controls (3.31,0.3) and (6.95,1.4) .. (10.93,3.29)   ;
%Curve Lines [id:da3028630630423357] 
\draw    (178,119) .. controls (116.31,110.04) and (105.11,170.39) .. (183.81,155.23) ;
\draw [shift={(185,155)}, rotate = 168.69] [color={rgb, 255:red, 0; green, 0; blue, 0 }  ][line width=0.75]    (10.93,-3.29) .. controls (6.95,-1.4) and (3.31,-0.3) .. (0,0) .. controls (3.31,0.3) and (6.95,1.4) .. (10.93,3.29)   ;
%Curve Lines [id:da6893767255986205] 
\draw    (283,250) .. controls (221.31,241.05) and (210.11,301.39) .. (288.81,286.23) ;
\draw [shift={(290,286)}, rotate = 168.69] [color={rgb, 255:red, 0; green, 0; blue, 0 }  ][line width=0.75]    (10.93,-3.29) .. controls (6.95,-1.4) and (3.31,-0.3) .. (0,0) .. controls (3.31,0.3) and (6.95,1.4) .. (10.93,3.29)   ;
%Curve Lines [id:da14033244462595684] 
\draw    (322.05,249.96) .. controls (388.99,248.97) and (388.37,296.9) .. (326,287) ;
\draw [shift={(320,250)}, rotate = 358.34] [color={rgb, 255:red, 0; green, 0; blue, 0 }  ][line width=0.75]    (10.93,-3.29) .. controls (6.95,-1.4) and (3.31,-0.3) .. (0,0) .. controls (3.31,0.3) and (6.95,1.4) .. (10.93,3.29)   ;

% Text Node
\draw (376,259) node [anchor=north west][inner sep=0.75pt]    {$\sigma _{H} f_{H}$};
% Text Node
\draw (206,260) node [anchor=north west][inner sep=0.75pt]    {$\sigma _{H}$};
% Text Node
\draw (284,192) node [anchor=north west][inner sep=0.75pt]    {$f_{A} \sigma _{A}$};
% Text Node
\draw (285,70) node [anchor=north west][inner sep=0.75pt]    {$\tau \sigma _{J}$};
% Text Node
\draw (62,125) node [anchor=north west][inner sep=0.75pt]    {$( 1-\tau ) \sigma _{J}$};
% Text Node
\draw (470,128) node [anchor=north west][inner sep=0.75pt]    {$\sigma _{A}$};
% Text Node
\draw (282,257.5) node [anchor=north west][inner sep=0.75pt]   [align=left] {Hosts};
% Text Node
\draw (394,126.5) node [anchor=north west][inner sep=0.75pt]   [align=left] {Adults };
% Text Node
\draw (151.5,127) node [anchor=north west][inner sep=0.75pt]   [align=left] {Juveniles};

\end{tikzpicture} \label{fig:my_label}
\end{center}
\end{figure}

The disease is not inherited in vectors, $i.e.$ infected adult triatomines
will give susceptible juveniles. However, it has been found that in mammalian hosts there is a certain vertical transmission rate which we will denote by $\nu$, see \cite{mahdjoub2020}. This will be the only process assumed to be generated by parental effects.\\

When a juvenile vector, respectively adult, is contaminated by an infected
host during the blood meal, the transmission rate of \textit{T. cruzi} is
noted $\Lambda _{J}$, respectively $\Lambda _{A}$. A susceptible host becomes infected via juvenile or adult vectors with a rate $\Lambda _{H}$. Generally, these transmission parameters are density-dependent and are written from the WAIFW matrix (Who Acquires Infection From Whom), see \cite{klepac2011}. For reasons of simplification, we will assume that these rates are constant in the interval $\left] 0,1\right[$, see Table \ref{tab:my_label} below.

\subsection{Habitat of vectors and their hosts}

The habitat of vectors and their hosts is a part of the forest represented
by $\Omega =\left[ -\ell,L\right] \times\left[ 0,1\right] $. It is
divided into two juxtaposed subdomains, $\Omega _{S}=\left[ 0,L\right] \times \left[ 0,\,1\right] $ where susceptible individuals live and $\Omega _{I}=\left[ -\ell,0\right] \times\left[ 0,1\right] $ where infected individuals live. The common boundary is $\Gamma =\left\{ 0\right\} \times\left[ 0,1\right]$.

\begin{figure}[ht]
\centering
\tikzset{every picture/.style={line width=0.75pt}} %set default line width to 0.75pt        
\begin{tikzpicture}[scale=0.8,x=0.75pt,y=0.75pt,yscale=-1,xscale=1]

%Shape: Rectangle [id:dp6769390042213541] 
\draw  [line width=1.5]  (158,63) -- (478,63) -- (478,191) -- (158,191) -- cycle ;
%Straight Lines [id:da009656036603675533] 
\draw [line width=1.5]    (309,64) -- (309,192) ;
%Straight Lines [id:da5418551142872541] 
\draw    (347,33) -- (313.21,77.41) ;
\draw [shift={(312,79)}, rotate = 307.27] [color={rgb, 255:red, 0; green, 0; blue, 0 }  ][line width=0.75]    (10.93,-3.29) .. controls (6.95,-1.4) and (3.31,-0.3) .. (0,0) .. controls (3.31,0.3) and (6.95,1.4) .. (10.93,3.29)   ;
% Text Node
\draw (348,16) node [anchor=north west][inner sep=0.75pt]    {$\Gamma $};
% Text Node
\draw (303,36) node [anchor=north west][inner sep=0.75pt]    {$1$};
% Text Node
\draw (303,202) node [anchor=north west][inner sep=0.75pt]    {$0$};
% Text Node
\draw (472,199) node [anchor=north west][inner sep=0.75pt]    {$L$};
% Text Node
\draw (148,198) node [anchor=north west][inner sep=0.75pt]    {$-\ell$};
% Text Node
\draw (382,109) node [anchor=north west][inner sep=0.75pt]    {$\Omega _{S}$};
% Text Node
\draw (229,109) node [anchor=north west][inner sep=0.75pt]    {$\Omega _{I}$};
\end{tikzpicture}
\caption{Habitat of vectors and their hosts}
\end{figure}
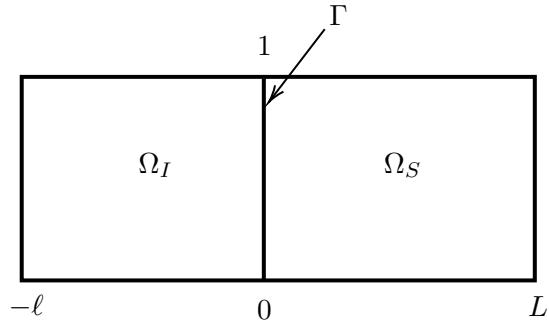

On $\Omega $, juvenile and adult vectors and their hosts diffuse
respectively with constant diffusion coefficients $d_{J}>0$, $d_{A}>0$ and $%
d_{H}>0$. 

More precisely, we denote $d_{J}^{+}$, $d_{A}^{+}$ and $d_{H}^{+}$ the
diffusion coefficients on $\Omega _{S}$ and $d_{J}^{-}$, $d_{A}^{-}$ et $%
d_{H}^{-}$ the diffusion coefficients on $\Omega _{I}$. The same notation
will be adopted for the demographic coefficients. 

Table \ref{tab:my_label} below contains the demographic diffusion and disease transmission coefficients defined on $\Omega _{S}$ and $\Omega _{I}$.

%$\underset{\text{{\Large Table 1: Definitions and properties of demographic,
%diffusion, vertical transmission  and disease process}}}{%
\begin{table}[ht]
    \centering
\begin{tabular}{|c|c|c|}
\hline
\textbf{Parameters} & \textbf{Definitions} & \textbf{Properties} \\ 
\hline 
$\sigma _{J}^{+}$, $\sigma _{J}^{-}$ & Probability of survival of juvenile
vectors & $0\leq \sigma _{J}^{+}$, $\sigma _{J}^{-}\leq 1$ \\ 
\hline
$\sigma _{A}^{+}$, $\sigma _{A}^{-}$ & Probability of survival of adult
vectors & $0\leq \sigma _{A}^{+}$, $\sigma _{A}^{-}\leq 1$ \\ 
\hline
$\sigma _{H}^{+}$, $\sigma _{H}^{-}$ & Probability of survival of hosts & $%
0\leq \sigma _{H}^{+}$, $\sigma _{H}^{-}\leq 1$ \\ 
\hline
$\tau ^{+}$, $\tau ^{-}$ & Transition probability from juvenile to adult
vectors & $0\leq \tau ^{+}$, $\tau ^{-}\leq 1$ \\ 
\hline
$f_{A}^{+}$, $f_{A}^{-}$ & Adult vector female fecundity & $f_{A}^{+}$, $%
f_{A}^{-}\geq 0$ \\ 
\hline
$f_{H}^{+}$, $f_{H}^{-}$ & Host female fecundity & $f_{H}^{+}$, $%
f_{H}^{-}\geq 0$ \\ 
\hline
$\nu $ & Vertical transmission rate per host female & $\nu \geq 0$ \\ 
\hline
$\Lambda _{J}$ & Infection rate of juvenile vectors & $\Lambda _{J}>0$ \\ 
\hline
$\Lambda _{A}$ & Infection rate of adult vectors & $\Lambda _{A}>0$ \\ 
\hline
$\Lambda _{H}$ & Infection rate of hosts & $\Lambda _{H}>0$ \\ 
\hline
$d_{J}^{+}$, $d_{J}^{-}$ & Diffusion coefficient of juvenile vectors & $%
d_{J}^{+}$, $d_{J}^{-}>0$ \\ 
\hline
$d_{A}^{+}$, $d_{A}^{-}$ & Diffusion coefficient of adult vectors & $%
d_{A}^{+}$, $d_{A}^{-}>0$ \\ 
\hline
$d_{H}^{+}$, $d_{H}^{-}$ & Diffusion coefficient of hosts & $d_{H}^{+}$, $%
d_{H}^{-}>0$ \\
\hline
\end{tabular}
%}$
    \caption{Definitions and properties of demographic diffusion vertical transmission and disease process}
    \label{tab:my_label}
\end{table}

\subsection{Skew Brownian motion across the interface $\Gamma$}

Following \cite{cantrell1998}, it is assumed that there
is a flux of juvenile vectors, adult vectors and host that cross the interface $\Gamma $, described by the asymmetric equations:

\begin{equation}\label{asym}
\left\{ 
\begin{array}{lll}
\dis p_{J}\,d_{J}^{-}\,\frac{\partial J_{I}}{\partial x}(t,x,y) & = & \dis%
(1-p_{J})\,d_{J}^{+}\,\frac{\partial J_{S}}{\partial x}(t,x,y) \\ 
\ecart\dis p_{A}\,d_{A}^{-}\,\frac{\partial A_{I}}{\partial x}(t,x,y) & = & %
\dis(1-p_{A})\,d_{A}^{+}\,\frac{\partial A_{S}}{\partial x}(t,x,y) \\ 
\ecart\dis p_{H}\,d_{H}^{-}\,\frac{\partial H_{I}}{\partial x}(t,x,y) & = & %
\dis(1-p_{H})\,d_{H}^{+}\,\frac{\partial H_{S}}{\partial x}(t,x,y),
\end{array}
\right. t>0,~(x,y)\in \Gamma  
\end{equation}
where $p_{J}$, $p_{A}$ and $p_{H} \in (0,1)$, are respectively the probabilities that a juvenile vector, an adult vector and a host, being in $\Omega
_{I}$ crosses the interface $\Gamma $ to be in $\Omega _{S}$.

Demographic, heredity and infection processes generate susceptible
individuals who are in the domain $\Omega _{I}$ of infected and infected
individuals who are in the domain $\Omega _{S}$ of susceptible.\\
During the diffusion process, the susceptible join $\Omega _{S}$ and the infected $\Omega _{I}$. The different flows that cross $\Gamma $ are written from $T_1$ to $T_5$.\\

Details of the movements are summarized in Table \ref{tab:my_label 2} below. We will assume that these are the only moves that take place across the boundary $\Gamma$.

\begin{table}[ht]
    \centering
\begin{tabular}{|c|c|l|}
\hline
\textbf{Flow} & \textbf{Direction} & \hspace{4cm} \textbf{Description} \\ 
\hline
$T_1$ & $\Omega _{I}$ $\rightarrow $ $\Omega _{S}$ & Susceptible juvenile vectors descended from infected adult vectors \\ 
\hline
$T_2$ & $\Omega _{I}$ $\rightarrow $ $\Omega _{S}$ & Susceptible hosts descended from infected hosts (no inheritance) \\ 
\hline
$T_3$ & $\Omega _{S}$ $\rightarrow $ $\Omega _{I}$ & Susceptible hosts that become infected \\ 
\hline
$T_4$ & $\Omega _{S}$ $\rightarrow $ $\Omega _{I}$ & Susceptible adult vectors that become infected \\ 
\hline
$T_5$ & $\Omega _{S}$ $\rightarrow $ $\Omega _{I}$ & Susceptible juvenile vectors that become infected \\
\hline
\end{tabular}

    \caption{Description of the different flows crossing the interface $\Gamma$}\label{tab:my_label 2}
\end{table}

These flows are known and are defined by: 
\begin{equation}
\left\{ 
\begin{array}{lll}

\ecart T_{1} & = & \dis f_{A}^{-}\,\sigma _{A}^{-}\,d_{A}^{-}\,\frac{%
\partial A_{I}}{\partial x}\left\vert _{\Gamma }\right. \\ 

\ecart T_{2} & = & \dis\left( 1-\nu \right) \,f_{H}^{-}\,\sigma
_{H}^{-}\,d_{H}^{-}\,\frac{\partial H_{I}}{\partial x}\left\vert _{\Gamma
}\right. \\ 

\ecart T_{3} & = & \dis\sigma _{H}^{+}\,d_{H}^{+}\,\frac{\partial H_{S}}{%
\partial x}\left\vert _{\Gamma }\right. +\sigma
_{H}^{+}\,f_{H}^{+}\,d_{H}^{+}\,\frac{\partial H_{S}}{\partial x}\left\vert
_{\Gamma }\right. +\left( 1-\nu \right) \,f_{H}^{-}\,\sigma
_{H}^{-}\,d_{H}^{-}\,\frac{\partial H_{I}}{\partial x}\left\vert _{\Gamma
}\right. \\ 

\ecart T_{4} & = & \dis\tau ^{+}\,\sigma _{J}^{+}\,d_{J}^{+}\,\frac{\partial
J_{S}}{\partial x}\left\vert _{\Gamma }\right. +\sigma _{A}^{+}\,d_{A}^{+}\,%
\frac{\partial A_{S}}{\partial x}\left\vert _{\Gamma }\right. \\ 

\ecart T_{5} & = & \dis\left( 1-\tau ^{+}\right) \,\sigma
_{J}^{+}\,d_{J}^{+}\,\frac{\partial J_{S}}{\partial x}\left\vert _{\Gamma
}\right. +f_{A}^{+}\,\sigma _{A}^{+}\,d_{A}^{+}\,\frac{\partial A_{S}}{%
\partial x}\left\vert _{\Gamma }\right. +T_{1}.%

\end{array}
\right.  \label{traces}
\end{equation}

\bigskip

These operators $T_{i}$, $i=1,2,3,4,5$, of traces, acting on the interface $%
\Gamma $, will compose the perturbation of the dispersal operator in the reaction-diffusion system written below.

\section{Reaction-diffusion system and its operational formulation}

The complete reaction-diffusion model considered in this paper is the following
\begin{equation}\label{Syst Reaction-diffusion}
\left\{ \begin{array}{llll}
\dis\frac{\partial J_{I}}{\partial t}&=& \dis d_{J}^{-}\Delta J_{I}+\left(\left( 1-\tau^{-}\right) \,\sigma _{J}^{-} - 1 \right)\,J_{I} + \sigma_A^- \,f_A^-\,A_I+ T_{5}\,\Lambda _{J}, & t>0,~(x,y)\in
\Omega _{I} \\ \ecart

\dis\frac{\partial A_{I}}{\partial t} &=& \dis d_{A}^{-}\Delta A_{I}+\tau
^{-}\,\sigma _{J}^{-}\,J_{I} + \left(\sigma _{A}^{-} - 1\right)\,A_{I}+T_{4}\,\Lambda_{A}, & t>0,~(x,y)\in \Omega _{I} \\ \ecart

\dis\frac{\partial H_{I}}{\partial t} &=& \dis d_{H}^{-}\Delta H_{I}+\left(\left(1 +\nu \,f_{H}^{-}\right)\,\sigma _{H}^{-} - 1\right)H_{I} + T_{3}\,\Lambda_{H}, & t>0,~(x,y)\in \Omega _{I} \\ \ecart

\dis\frac{\partial J_{S}}{\partial t} &=& \dis d_{J}^{+}\Delta J_{S} + \left(
\left( 1-\tau ^{+}\right) \,\sigma _{J}^{+} - 1\right)J_{S}+f_{A}^{+}\,\sigma
_{A}^{+}\,A_{S}+T_{1}, & t>0,~(x,y)\in \Omega _{S} \\ \ecart

\dis\frac{\partial A_{S}}{\partial t} &=& \dis d_{A}^{+}\Delta A_{S} + \tau ^{+}\,\sigma _{J}^{+}\,J_{S} + \left(\sigma _{A}^{+} - 1 \right) A_{S}, & t>0,~(x,y)\in \Omega _{S} \\ \ecart

\dis\frac{\partial H_{S}}{\partial t} &=& \dis d_{H}^{+}\Delta H_{S} + \left(\left(1 + f_{H}^{+}\right)\sigma _{H}^{+} - 1 \right) H_{S} +T_{2}, & t>0,~(x,y)\in \Omega _{S},
\end{array}\right.  
\end{equation}
with the initial conditions
\begin{equation}
\left\{ \hspace{-0.1cm}%
\begin{array}{llll}
J_{I}(0,x,y) & = & J_{I}^{0}(x,y), &  \\ 
A_{I}(0,x,y) & = & A_{I}^{0}(x,y), & (x,y)\in \Omega _{I} \\ 
H_{I}(0,x,y) & = & H_{I}^{0}(x,y), & 
\end{array}%
\right. \quad\left\{ \hspace{-0.1cm}%
\begin{array}{llll}
J_{S}(0,x,y) & = & J_{S}^{0}(x,y), &  \\ 
A_{S}(0,x,y) & = & A_{S}^{0}(x,y), & (x,y)\in \Omega _{S}, \\ 
H_{S}(0,x,y) & = & H_{S}^{0}(x,y), & 
\end{array}%
\right.  \label{CI 2}
\end{equation}
the boundary conditions
\begin{equation}
\left\{ 
\begin{array}{llll}
(J_{I}(t,\sigma _{1},\sigma _{2}),A_{I}(t,\sigma _{1},\sigma
_{2}),H_{I}(t,\sigma _{1},\sigma _{2})) & = & (0,0,0), & t>0,~(\sigma
_{1},\sigma _{2})\in \partial \Omega _{I}\setminus \Gamma \\ 
\ecart(J_{S}(t,\sigma _{1},\sigma _{2}),A_{S}(t,\sigma _{1},\sigma
_{2}),H_{S}(t,\sigma _{1},\sigma _{2})) & = & (0,0,0), & t>0,~(\sigma
_{1},\sigma _{2})\in \partial \Omega _{S}\setminus \Gamma ,%
\end{array}%
\right.  \label{CL 2}
\end{equation}
and the transmission conditions
\begin{equation}\label{CT2}
\left\{ 
\begin{array}{llll}
J_{I}(t,0,y) & = & J_{S}(t,0,y), & t>0,~y\in (0,1) \\ 
\ecart A_{I}(t,0,y) & = & A_{S}(t,0,y), & t>0,~y\in (0,1) \\ 
\ecart H_{I}(t,0,y) & = & H_{S}(t,0,y), & t>0,~y\in (0,1) \\ 
&  &  &  \\ 
\dis p_{J}\,d_{J}^{-}\,\frac{\partial J_{I}}{\partial x}(t,0,y) & = & \dis (1-p_{J})\,d_{J}^{+}\,\frac{\partial J_{S}}{\partial x}(t,0,y), & t>0,~y \in (0,1) \\ 
\ecart\dis p_{A}\,d_{A}^{-}\,\frac{\partial A_{I}}{\partial x}(t,0,y) & = & \dis(1-p_{A})\,d_{A}^{+}\,\frac{\partial A_{S}}{\partial x}(t,0,y), & t>0,~y \in (0,1) \\ 
\ecart\dis p_{H}\,d_{H}^{-}\,\frac{\partial H_{I}}{\partial x}(t,0,y) & = & \dis (1-p_{H})\,d_{H}^{+}\,\frac{\partial H_{S}}{\partial x}(t,0,y), & t>0,~y \in (0,1).
\end{array}
\right.  
\end{equation}

The boundary conditions \eqref{CL 2} reflect the fact that the entire
population evolves inside the domain $\Omega $ and that the outside is a
hostile habitat. The first three equations in \eqref{CT2} express the continuity of the population density of individuals at $\Gamma$, while the last three equations express the skew Brownian motion represented by the probabilities $p_{J}$, $p_{A}$ and $p_{H}$ which can be different from $1/2$. This implies, in particular, that the flows through the interface $\Gamma $ are not continuous.

We set
\begin{equation*}
\mu _{I}=p_{J}\,d_{J}^{-},\quad \alpha _{I}=p_{A}\,d_{A}^{-}\quad \text{and}%
\quad \beta _{I}=p_{H}\,d_{H}^{-},
\end{equation*}%
and 
\begin{equation*}
\mu _{S}=(1-p_{J})\,d_{J}^{+},\quad \alpha _{S}=(1-p_{A})\,d_{A}^{+}\quad 
\text{and}\quad \beta _{S}=(1-p_{H})\,d_{H}^{+}.
\end{equation*}
Note that $\mu_I$, $\alpha_I$, $\beta_I$ and $\mu_S$, $\alpha_S$, $\beta_S$ are strictly positive.

We will write system \eqref{Syst Reaction-diffusion} in the form of a
perturbed abstract Cauchy problem, see \eqref{Pb evolution} below. To this end, we will adopt the following abstract matrix notations
\begin{equation*}
\begin{array}{lll}
V(t)(.):=V(t,.) & = & \dis\left( 
\begin{array}{c}
J(t) \\ 
A(t) \\ 
H(t)%
\end{array}%
\right) (.):=\left( 
\begin{array}{c}
J(t,.) \\ 
A(t,.) \\ 
H(t,.)%
\end{array}%
\right) \\ \ecart
 & := & \dis\left\{ 
\begin{array}{ll}
V_{I}(t,.)=\left( 
\begin{array}{c}
J_{I}(t,.) \\ 
A_{I}(t,.) \\ 
H_{I}(t,.)%
\end{array}%
\right) & \text{on }\Omega _{I} \\ 
&  \\ 
V_{S}(t,.)=\left( 
\begin{array}{c}
J_{S}(t,.) \\ 
A_{S}(t,.) \\ 
H_{S}(t,.)%
\end{array}%
\right) & \text{on }\Omega _{S}.%
\end{array}%
\right.%
\end{array}%
\end{equation*}
Thus, system \eqref{Syst Reaction-diffusion} is written on $\Omega _{I}$ as 
\begin{equation*}
\begin{array}{lll}
V'_{I}(t) = \medskip \\ 
\dis\left( \begin{array}{ccc}
d_{J}^{-}\Delta +[(1-\tau ^{-})\sigma _{J}^{-}-1]I & 0 & 0 \\ 
0 & d_{A}^{-}\Delta + (\sigma _{A}^{-} - 1)I & 0 \\ 
0 & 0 & d_{H}^{-}\Delta + [\sigma _{H}^{-}(1+\nu f_{H}^{-}) - 1]I
\end{array}
\right) \left( 
\begin{array}{c}
J_{I}(t) \\ 
A_{I}(t) \\ 
H_{I}(t)
\end{array}
\right) \medskip \\  
\dis +\left(\begin{array}{c}
\dis \sigma_A^-\,f_A^-\,A_I + \Lambda _{J}\left( \left( 1-\tau ^{+}\right) \,\sigma
_{J}^{+}\,d_{J}^{+}\,\frac{\partial J_{S}(t)}{\partial x}\left\vert _{\Gamma
}\right. +f_{A}^{+}\,\sigma _{A}^{+}\,d_{A}^{+}\,\frac{\partial A_{S}(t)}{
\partial x}\left\vert _{\Gamma }\right. +f_{A}^{-}\,\sigma
_{A}^{-}\,d_{A}^{-}\,\frac{\partial A_{I}(t)}{\partial x}\left\vert _{\Gamma
}\right. \right) \\ 
\ecart\dis\tau ^{-}\,\sigma _{J}^{-}\,J_{I}(t)+\Lambda _{A}\left( \tau
^{+}\,\sigma _{J}^{+}\,d_{J}^{+}\,\frac{\partial J_{S}(t)}{\partial x}%
\left\vert _{\Gamma }\right. +\sigma _{A}^{+}\,d_{A}^{+}\,\frac{\partial
A_{S}(t)}{\partial x}\left\vert _{\Gamma }\right. \right) \\ 
\ecart\dis\Lambda _{H}\left( \sigma _{H}^{+}\,d_{H}^{+}\,\frac{\partial
H_{S}(t)}{\partial x}\left\vert _{\Gamma }\right. +\sigma
_{H}^{+}\,f_{H}^{+}\,d_{H}^{+}\,\frac{\partial H_{S}(t)}{\partial x}
\left\vert _{\Gamma }\right. +\left( 1-\nu \right) \,f_{H}^{-}\,\sigma
_{H}^{-}\,d_{H}^{-}\,\frac{\partial H_{I}(t)}{\partial x}\left\vert _{\Gamma
}\right. \right)
\end{array}\right),
\end{array}
\end{equation*}
and on $\Omega _{S}$ as
\begin{equation*}
\begin{array}{l}
V_{S}'(t) = \\ \ecart
\dis \left( 
\begin{array}{ccc}
d_{J}^{+}\Delta + [(1-\tau ^{+})\sigma _{J}^{+} - 1]I & 0 & 0 \\ 

0 & d_{A}^{+}\Delta + (\sigma _{A}^{+} - 1)I & 0 \\ 

0 & 0 & d_{H}^{+}\Delta + [\sigma _{H}^{+}(1+f_{H}^{+}) - 1]I
\end{array}
\right) \hspace{-0.15cm}\left( 
\begin{array}{c}
J_{S}(t) \\ 
A_{S}(t) \\ 
H_{S}(t)
\end{array}
\right) \\ 
  \\ 
 \dis+\left( 
\begin{array}{c}
\dis f_{A}^{+}\,\sigma_{A}^{+}\,A_{S}(t)+f_{A}^{-}\,\sigma _{A}^{-}\,d_{A}^{-}\,\frac{\partial A_{I}(t)}{\partial x}\left\vert _{\Gamma }\right. \\ \ecart

\dis \tau ^{+}\,\sigma _{J}^{+}\,J_{S}(t) \\ \ecart

\dis (1-\nu) \,f_{H}^{-}\,\sigma_{H}^{-}\,d_{H}^{-}\,\frac{\partial H_{I}(t)}{\partial x}\left\vert _{\Gamma}\right.
\end{array}
\right) .%
\end{array}%
\end{equation*}
Then, equations \eqref{Syst Reaction-diffusion}, \eqref{CI 2}, \eqref{CL 2} and \eqref{CT2} are written in the following form
\begin{equation}
\left\{ 
\begin{array}{lll}
V^{\prime }(t) & = & \L V(t)+\B V(t) \\ 
\ecart V(0) & = & V_{0},
\end{array}
\right.  \label{Pb evolution}
\end{equation}
where operator $\L$ will act only with respect to the spatial variables $x$
and $y$ and is defined by 
\begin{equation*}
\left\{ \hspace{-0.1cm}%
\begin{array}{cll}
D(\L ) & \hspace{-0.1cm}= & \hspace{-0.2cm}\left\{ \hspace{-0.15cm}%
\begin{array}{l}
\varphi =\left( 
\begin{array}{c}
j \\ 
a \\ 
h%
\end{array}%
\right) \in \left( L^{p}(\Omega )\right) ^{3}:\dis\varphi _{|_{\Omega
_{I}}}=\left( 
\begin{array}{c}
j_{I} \\ 
a_{I} \\ 
h_{I}%
\end{array}%
\right) \in \left(W^{2,p}(\Omega _{I})\right) ^{3},~\dis\varphi _{|_{\Omega _{S}}}=\left( 
\begin{array}{c}
j_{S} \\ 
a_{S} \\ 
h_{S}%
\end{array}%
\right) \in \left(W^{2,p}(\Omega _{S})\right) ^{3}, \\ 
\ecart\dis\left( 
\begin{array}{c}
j_{I} \\ 
a_{I} \\ 
h_{I}%
\end{array}%
\right) (-\ell ,y)=\left( 
\begin{array}{c}
j_{I} \\ 
a_{I} \\ 
h_{I}%
\end{array}%
\right) (x,0)=\left( 
\begin{array}{c}
j_{I} \\ 
a_{I} \\ 
h_{I}%
\end{array}%
\right) (x,1)=\left( 
\begin{array}{c}
0 \\ 
0 \\ 
0%
\end{array}%
\right) ,~x\in \lbrack -\ell ,0],~y\in \lbrack 0,1], \\ 
\ecart\dis\left( 
\begin{array}{c}
j_{S} \\ 
a_{S} \\ 
h_{S}%
\end{array}%
\right) (L,y)=\left( 
\begin{array}{c}
j_{S} \\ 
a_{S} \\ 
h_{S}%
\end{array}%
\right) (x,0)=\left( 
\begin{array}{c}
j_{S} \\ 
a_{S} \\ 
h_{S}%
\end{array}%
\right) (x,1)=\left( 
\begin{array}{c}
0 \\ 
0 \\ 
0%
\end{array}%
\right) ,~x\in \lbrack 0,L],~y\in \lbrack 0,1], \\ 
\ecart\dis\left( 
\begin{array}{c}
j_{I} \\ 
\ecart a_{I} \\ 
\ecart h_{I}%
\end{array}%
\right) (0,y)=\left( 
\begin{array}{c}
j_{S} \\ 
\ecart a_{S} \\ 
\ecart h_{S}%
\end{array}%
\right) (0,y)~~\text{and}~~\left( 
\begin{array}{c}
\mu _{I}\dfrac{\partial j_{I}}{\partial x} \\ 
\ecart\alpha _{I}\dfrac{\partial a_{I}}{\partial x} \\ 
\ecart\beta _{I}\dfrac{\partial h_{I}}{\partial x}%
\end{array}%
\right) (0,y)=\left( 
\begin{array}{c}
\mu _{S}\dfrac{\partial j_{S}}{\partial x} \\ 
\ecart\alpha _{S}\dfrac{\partial a_{S}}{\partial x} \\ 
\ecart\beta _{S}\dfrac{\partial h_{S}}{\partial x}%
\end{array}%
\right) (0,y),~y\in \lbrack 0,1]%
\end{array}%
\hspace{-0.1cm}\right\} \\ 
&  &  \\ 
\L \varphi & \hspace{-0.1cm}= & \hspace{-0.2cm}\left\{ 
\begin{array}{ll}
\L _{I}\varphi & \text{on }\Omega _{I} \\ 
\ecart\L _{S}\varphi & \text{on }\Omega _{S},
\end{array}%
\right.%
\end{array}%
\right.
\end{equation*}
with 
\begin{equation*}
\L _{I}:=\left( \begin{array}{ccc}
d_{J}^{-}\Delta + [(1-\tau ^{-})\sigma _{J}^{-} - 1]I & 0 & 0 \\ 
0 & d_{A}^{-}\Delta + (\sigma _{A}^{-} - 1)I & 0 \\ 
0 & 0 & d_{H}^{-}\Delta + [\sigma _{H}^{-}(1+\nu f_{H}^{-}) - 1]I
\end{array}\right),
\end{equation*}
and 
\begin{equation*}
\L _{S}:=\left( \begin{array}{ccc}
d_{J}^{+}\Delta + [(1-\tau ^{+})\sigma _{J}^{+} - 1] I & 0 & 0 \\ 
0 & d_{A}^{+}\Delta + (\sigma _{A}^{+} - 1)I & 0 \\ 
0 & 0 & d_{H}^{+}\Delta + [\sigma _{H}^{+}(1+f_{H}^{+}) - 1]I
\end{array}\right).
\end{equation*}
We have noted all the functions depending on $x$ and $y$ in lower case.
\begin{Rem} 
The domain of $\L$ takes into account, on the one hand, the dispersal process of the \textit{T. Cruzi} parasite in the domains $\Omega_I$, $\Omega_S$ and on the other hand, the boundary and transmission conditions. Its action is defined by those of $\L_I$ and $\L_S$.

\end{Rem}
Likewise, we set \begin{equation*}
\left\{ \hspace{-0.1cm}
\begin{array}{cll}
D(\B) & = & D(\L ) \\ 
\ecart\B\varphi & \hspace{-0.1cm}= & \hspace{-0.2cm}\left\{ 
\begin{array}{ll}
\B_{I}\varphi & \text{on }\Omega _{I} \\ 
\ecart\B_{S}\varphi & \text{on }\Omega _{S},
\end{array}
\right.
\end{array}\right.
\end{equation*}
where, for $x\in [-\ell ,0]$ and $y\in [0,1]$, $\B_{I}$ is defined by
\begin{equation*}
\begin{array}{l}
\left( \B_{I}\varphi \right) (x,y) = \medskip \\
\left( \begin{array}{c}
\dis \sigma_A^-\,f_A^-\,a_I(x,y) + \Lambda _{J}\left( \left( 1-\tau ^{+}\right) \,\sigma_{J}^{+}\,d_{J}^{+}\,\frac{\partial j_{S}}{\partial x}(0,y)+f_{A}^{+}\,\sigma _{A}^{+}\,d_{A}^{+}\,\frac{\partial a_{S}}{\partial x}(0,y)+f_{A}^{-}\,\sigma _{A}^{-}\,d_{A}^{-}\,\frac{\partial a_{I}}{\partial x
}(0,y)\right) \\ \ecart

\dis\tau ^{-}\,\sigma _{J}^{-}\,j_{I}(x,y)+\Lambda _{A}\left( \tau^{+}\,\sigma _{J}^{+}\,d_{J}^{+}\,\frac{\partial j_{S}}{\partial x}(0,y)+\sigma _{A}^{+}\,d_{A}^{+}\,\frac{\partial a_{S}}{\partial x}(0,y)\right) \\ \ecart

\dis\Lambda _{H}\left( \sigma _{H}^{+}\,d_{H}^{+}\,\frac{\partial h_{S}}{\partial x}(0,y)+\sigma _{H}^{+}\,f_{H}^{+}\,d_{H}^{+}\,\frac{\partial
h_{S}}{\partial x}(0,y)+\left( 1-\nu \right) \,f_{H}^{-}\,\sigma_{H}^{-}\,d_{H}^{-}\,\frac{\partial h_{I}}{\partial x}(0,y)\right)
\end{array}\right) ,
\end{array}
\end{equation*}
and for $x\in [0,L]$ and $y\in [0,1]$, $\B_{S}$ is defined by
\begin{equation*}
\left( \B_{S}\varphi \right) (x,y)=\left( 
\begin{array}{c}
\dis f_{A}^{+}\,\sigma_{A}^{+}\,a_{S}(x,y)+f_{A}^{-}\,\sigma _{A}^{-}\,d_{A}^{-}\,\frac{\partial a_{I}}{\partial x}(0,y) \\ \ecart

\dis \tau ^{+}\,\sigma _{J}^{+}\,j_{S}(x,y) \\ \ecart

\dis (1-\nu)\,f_{H}^{-}\,\sigma_{H}^{-}\,d_{H}^{-}\,\frac{\partial h_{I}}{\partial x}(0,y)
\end{array}\right).
\end{equation*}
Moreover, due to the transmission conditions, defined in $\D(\L)$, we
obtain
\begin{equation*}
\begin{array}{l}
\left( \B_{I}\varphi \right) (x,y) = \medskip \\
\left( \begin{array}{c}
\dis \sigma_A^-\,f_A^-\,a_I(x,y) + \Lambda _{J}\left( \frac{\left( 1-\tau ^{+}\right) \,\sigma_{J}^{+}\,d_{J}^{+}\mu _{I}}{\mu _{S}}\,\frac{\partial j_{I}}{\partial x}(0,y)+\left( f_{A}^{-}\,\sigma _{A}^{-}\,d_{A}^{-}+f_{A}^{+}\,\sigma_{A}^{+}d_{A}^{+}\,\frac{\alpha _{I}}{\alpha _{S}}\right) \frac{\partial a_{I}}{\partial x}(0,y)\right) \\ \ecart

\dis\tau ^{-}\,\sigma _{J}^{-}\,j_{I}(x,y)+\Lambda _{A}\left( \frac{\tau ^{+}\,\sigma _{J}^{+}\,d_{J}^{+}\,\mu _{I}}{\mu _{S}}\,\frac{\partial j_{I}}{\partial x}(0,y)+\frac{\sigma _{A}^{+}\,d_{A}^{+}\,\alpha _{I}}{\alpha _{S}}\,\frac{\partial a_{I}}{\partial x}(0,y)\right) \\ \ecart

\dis\Lambda _{H}\,\left( \sigma _{H}^{+}\,d_{H}^{+}\,\frac{\left(1+f_{H}^{+}\right) \beta _{I}}{\beta _{S}}+\left( 1-\nu \right) \,\sigma_{H}^{-}\,f_{H}^{-}\,d_{H}^{-}\right) \frac{\partial h_{I}}{\partial x}(0,y)
\end{array}\right),
\end{array}
\end{equation*}
and 
\begin{equation*}
\left( \B_{S}\varphi \right) (x,y) = \left( 
\begin{array}{c}
\dis f_{A}^{+}\,\sigma _{A}^{+}\,a_{S}(x,y)+\frac{f_{A}^{-}\,\sigma _{A}^{-}\,d_{A}^{-}\,\alpha _{S}}{\alpha _{I}}\,\frac{\partial a_{S}}{\partial x}(0,y) \\ \ecart

\dis \tau ^{+}\,\sigma _{J}^{+}\,j_{S}(x,y) \\ \ecart

\dis (1-\nu) \,\frac{f_{H}^{-}\,\sigma_{H}^{-}\,d_{H}^{-}\,\beta _{S}}{\beta _{I}}\,\frac{\partial h_{S}}{\partial x}(0,y)
\end{array}\right).
\end{equation*}

\begin{Rem}
It is necessary to take $D(\B)=D(\L )$ in order to give a sense to the
traces of all the partial derivatives with respect to $x$ on $\Gamma$. Indeed, for example, for $j_{I}\in W^{2,p}(\Omega _{I})$, we have $%
\dfrac{\partial j_{I}}{\partial x}\in W^{1,p}(\Omega _{I})$ and then $\dfrac{%
\partial j_{I}}{\partial x}\left\vert _{\Gamma }\right. \in W^{1-\frac{1}{p}%
,p}(\Gamma )$ (see \cite{grisvard1969}, Theorem 5.10, p. 335), where%
\begin{equation*}
W^{1-\frac{1}{p},p}(\Gamma ):=\left\{ f\in L^{p}(\Gamma
):(s_{1},s_{2})\longmapsto \frac{f(s_{1})-f(s_{2})}{|s_{1}-s_{2}|}\in
L^{p}(\Gamma \times \Gamma )\right\} ,
\end{equation*}
\thinspace see \cite{triebel1978}, equation (8), p. 323 and \cite{grisvard1969}, p. 333. The same is true for the other partial derivatives
with respect to $x$ which appear in the action of $\B$.
\end{Rem}

In this paper, we will focus ourselves on the study of the spectral properties of operator $\L$ in order to prove the generation of analytic semigroup. To this end, throughout this work we assume that there exists $r_{0}>0$ such that 
\begin{equation}\label{hyp const}
\max \left( \dfrac{\sigma _{H}^{-}(1+\nu f_{H}^{-})-1}{d_{H}^{-}},\dfrac{\sigma _{H}^{+}(1+f_{H}^{+})-1}{d_{H}^{+}}\right)\leqslant r_{0}<\pi ^{2}.
\end{equation}
\begin{Rem}    

{
Condition (\ref{hyp const}) is a general condition for the operator $\L$ to generate an analytic semigroup.\\
Under the more realistic hypothesis that the disease does not affect the demographic parameters of the vectors and hosts (see \cite{mahdjoub2020}), condition (\ref{hyp const}) is written as:
   \begin{equation*}
    \dfrac{\sigma(1+f) - 1}{d} \leq r_0. 
      \end{equation*}
    where $\sigma = \sigma^+ = \sigma^-$, $f = f^+ = f^-$ and $d = d^+ = d^-$.
This inequality is not satisfied for the host population with a high survival probability and important fertility rate in parallel with low mobility. This is the case for a population that exhibits high density. The model is therefore applicable to the case of a host population that exhibits density dependence, which is generally the case in field situations.
}
\end{Rem}

\begin{Rem}\hfill
\begin{enumerate}
\item Note that if $\nu=0$, assumption \eqref{hyp const} becomes 
$$\dfrac{\sigma _{H}^{+}(1+f_{H}^{+})-1}{d_{H}^{+}}\leqslant r_{0}<\pi ^{2}.$$

\item When $\nu >0$, assumption \eqref{hyp const} is obviously satisfied if
$$\sigma^-_H\,\nu \,f_H^- \leqslant 1 - \sigma^-_H \quad\text{and}\quad \sigma^+_H \,f_H^+ \leqslant 1 - \sigma^+_H.$$

\item Note that this assumption \eqref{hyp const} only takes into account the parameters governing the density of hosts.
\end{enumerate}
\end{Rem}

Our main result in this paper is the following.
\begin{Th}\label{Th principal}
Assume that \eqref{hyp const} holds. Then, the main dispersal operator $\L $ is linear, closed, with dense domain and generates an analytic semigroup in $\E=\left[ L^{p}(\Omega )\right] ^{3}$, where $p \in (1,+\infty)$.
\end{Th}

\section{The spectral equation}

The spectral study of $\L $ is based on the study of the following equation: 
\begin{equation}
\L \varphi -\lambda \varphi =\psi \in \E:=(L^{p}(\Omega ))^{3},
\label{eq L phi - lambda phi = psi}
\end{equation}%
where $\varphi \in D(\L )$ and $\lambda $ is a complex number belonging to a sector which will be specified later.

The space $\E$ is normed by 
\begin{equation*}
\left\Vert \left( 
\begin{array}{c}
j \\ 
a \\ 
h
\end{array}
\right) \right\Vert _{\E}=\max \left( \Vert j\Vert _{L^{p}(\Omega )},\Vert
a\Vert _{L^{p}(\Omega )},\Vert h\Vert _{L^{p}(\Omega )}\right) ,
\end{equation*}%
which is equivalent to 
\begin{equation*}
\max \left( \Vert j_{I}\Vert _{L^{p}(\Omega _{I})}+\Vert j_{S}\Vert
_{L^{p}(\Omega _{S})},\Vert a_{I}\Vert _{L^{p}(\Omega _{I})}+\Vert
a_{S}\Vert _{L^{p}(\Omega _{S})},\Vert h_{I}\Vert _{L^{p}(\Omega
_{I})}+\Vert h_{S}\Vert _{L^{p}(\Omega _{S})}\right) .
\end{equation*}%
Equation \eqref{eq L phi - lambda phi = psi} writes 
\begin{equation*}
\L \left( 
\begin{array}{c}
j \\ 
a \\ 
h
\end{array}
\right) -\lambda \left( 
\begin{array}{c}
j \\ 
a \\ 
h
\end{array}
\right) =\left( \begin{array}{c}
k \\ 
m \\ 
n
\end{array}\right) ,
\end{equation*}
which gives the explicit following system
\begin{equation*}
\left\{ 
\begin{array}{ll}
\begin{array}{lll}
d_{J}^{-}\,\Delta j_{I}-\left( \lambda +1 -(1-\tau ^{-})\sigma _{J}^{-}\right)
j_{I} & = & k_{I} \\ 
\ecart d_{A}^{-}\,\Delta a_{I}-\left( \lambda + 1 -\sigma _{A}^{-}\right) a_{I} & = & m_{I} \\ \ecart 
d_{H}^{-}\,\Delta h_{I}-\left( \lambda + 1 - \sigma _{H}^{-}(1+\nu f_{H}^{-})\right) h_{I} & = & n_{I},
\end{array}
& \text{on }\Omega _{I} \\ 
&  \\ 
\begin{array}{lll}
d_{J}^{+}\,\Delta j_{S}-\left( \lambda + 1 - (1-\tau ^{+})\sigma
_{J}^{+}\right) j_{S} & = & k_{S} \\ 
\ecart d_{A}^{+}\,\Delta a_{S}-\left( \lambda + 1 - \sigma _{A}^{+}\right) a_{S} & = & m_{S} \\ \ecart 
d_{H}^{+}\,\Delta h_{S}-\left( \lambda + 1 - \sigma_{H}^{+}(1+f_{H}^{+})\right) h_{S} & = & n_{s},
\end{array}
& \text{on }\Omega _{S}
\end{array}
\right.
\end{equation*}
with 
\begin{equation*}
\left\{ 
\begin{array}{ll}
\begin{array}{lllll}
j_{I}(x,0) & = & j_{I}(x,1) & = & 0 \\ 
a_{I}(x,0) & = & a_{I}(x,1) & = & 0 \\ 
h_{I}(x,0) & = & h_{I}(x,1) & = & 0%
\end{array}
& x\in \lbrack -\ell ,0], \\ 
&  \\ 
\begin{array}{lllll}
j_{S}(x,0) & = & j_{S}(x,1) & = & 0 \\ 
a_{S}(x,0) & = & a_{S}(x,1) & = & 0 \\ 
h_{S}(x,0) & = & h_{S}(x,1) & = & 0%
\end{array}
& x\in \lbrack 0,L], \\ 
&  \\ 
\begin{array}{lllllll}
j_{I}(-\ell ,y) & = & a_{I}(-\ell ,y) & = & h_{I}(-\ell ,y) & = & 0 \\ 
j_{S}(L,y) & = & a_{S}(L,y) & = & h_{S}(L,y) & = & 0,
\end{array}
& y\in \lbrack 0,1],%
\end{array}%
\right.
\end{equation*}%
and 
\begin{equation*}
\left\{ 
\begin{array}{ll}
\begin{array}{lll}
j_{I}(0,y) & = & j_{S}(0,y) \\ 
\ecart a_{I}(0,y) & = & a_{S}(0,y) \\ 
\ecart h_{I}(0,y) & = & h_{S}(0,y)%
\end{array}
& y\in \lbrack 0,1] \\ 
&  \\ 
\begin{array}{lll}
\mu _{I}\dfrac{\partial j_{I}}{\partial x}(0,y) & = & \mu _{S}\dfrac{%
\partial j_{S}}{\partial x}(0,y) \\ 
\ecart\alpha _{I}\dfrac{\partial a_{I}}{\partial x}(0,y) & = & \alpha _{S}%
\dfrac{\partial a_{S}}{\partial x}(0,y) \\ 
\ecart\beta _{I}\dfrac{\partial h_{I}}{\partial x}(0,y) & = & \beta _{S}%
\dfrac{\partial h_{S}}{\partial x}(0,y).%
\end{array}
& y\in \lbrack 0,1]%
\end{array}%
\right.
\end{equation*}%
We then obtain the following three subsystems 
\begin{equation}\label{Pb trans en J}
\left\{ 
\begin{array}{l}
\begin{array}{llll}
\Delta j_{I}-\left( \dfrac{\lambda }{d_{J}^{-}} + \dfrac{1-(1-\tau ^{-})\,\sigma_{J}^{-}}{d_{J}^{-}}\right) j_{I} & = & \dfrac{k_{I}}{d_{J}^{-}} & \text{on }
\Omega _{I} \\  \ecart

\Delta j_{S}-\left( \dfrac{\lambda }{d_{J}^{+}}+\dfrac{1-(1-\tau
^{+})\,\sigma _{J}^{+}}{d_{J}^{+}}\right) j_{S} & = & \dfrac{k_{S}}{d_{J}^{+}} & \text{on }\Omega _{S}
\end{array} \\ \ecart
\begin{array}{llllll}
j_{I}(x,0) & = & j_{I}(x,1) & = & 0 & x\in \lbrack -\ell ,0] \\ 
j_{S}(x,0) & = & j_{S}(x,1) & = & 0 & x\in \lbrack 0,L] \\ 
j_{I}(-\ell ,y) & = & j_{S}(L,y) & = & 0 & y\in \lbrack 0,1] \\ 
j_{I}(0,y) & = & j_{S}(0,y) &  &  & y\in \lbrack 0,1] \\ 
\mu _{I}\dfrac{\partial j_{I}}{\partial x}(0,y) & = & \mu _{S}\dfrac{
\partial j_{S}}{\partial x}(0,y) &  &  & y\in \lbrack 0,1],
\end{array}
\end{array}
\right.  
\end{equation}

\begin{equation}\label{Pb trans en A}
\left\{ 
\begin{array}{l}
\begin{array}{llll}
\Delta a_{I}-\left( \dfrac{\lambda }{d_{A}^{-}}+\dfrac{1-\sigma _{A}^{-}}{
d_{A}^{-}}\right) a_{I} & = & \dfrac{m_{I}}{d_{A}^{-}} & \text{on }\Omega
_{I} \\ 
\ecart\Delta a_{S}-\left( \dfrac{\lambda }{d_{A}^{+}}+\dfrac{1-\sigma
_{A}^{+}}{d_{A}^{+}}\right) a_{S} & = & \dfrac{m_{S}}{
d_{A}^{+}} & \text{on }\Omega _{S}
\end{array}
\\ \ecart
\begin{array}{llllll}
a_{I}(x,0) & = & a_{I}(x,1) & = & 0 & x\in \lbrack -\ell ,0] \\ 
a_{S}(x,0) & = & a_{S}(x,1) & = & 0 & x\in \lbrack 0,L] \\ 
a_{I}(-\ell ,y) & = & a_{S}(L,y) & = & 0 & y\in \lbrack 0,1] \\ 
a_{I}(0,y) & = & a_{S}(0,y) &  &  & y\in \lbrack 0,1] \\ 
\alpha _{I}\dfrac{\partial a_{I}}{\partial x}(0,y) & = & \alpha _{S}\dfrac{%
\partial a_{S}}{\partial x}(0,y) &  &  & y\in \lbrack 0,1].%
\end{array}
\end{array}
\right.  
\end{equation}
and
\begin{equation}\label{Pb trans en H}
\left\{ 
\begin{array}{l}
\begin{array}{llll}
\Delta h_{I}-\left( \dfrac{\lambda }{d_{H}^{-}}+\dfrac{1-\sigma
_{H}^{-}\,(1+\nu f_{H}^{-})}{d_{H}^{-}}\right) h_{I} & = & \dfrac{n_{I}}{
d_{H}^{-}} & \text{on }\Omega _{I} \\ \ecart
\Delta h_{S}-\left( \dfrac{\lambda }{d_{H}^{+}}+\dfrac{1-\sigma
_{H}^{+}\,(1+f_{H}^{+})}{d_{H}^{+}}\right) h_{S} & = & 
\dfrac{n_{S}}{d_{H}^{+}} & \text{on }\Omega _{S}
\end{array} \\ \ecart
\begin{array}{llllll}
h_{I}(x,0) & = & h_{I}(x,1) & = & 0 & x\in \lbrack -\ell ,0] \\ 
h_{S}(x,0) & = & h_{S}(x,1) & = & 0 & x\in \lbrack 0,L] \\ 
h_{I}(-\ell ,y) & = & h_{S}(L,y) & = & 0 & y\in \lbrack 0,1] \\ 
h_{I}(0,y) & = & h_{S}(0,y) &  &  & y\in \lbrack 0,1] \\ 
\beta _{I}\dfrac{\partial h_{I}}{\partial x}(0,y) & = & \beta _{S}\dfrac{%
\partial h_{S}}{\partial x}(0,y) &  &  & y\in \lbrack 0,1].%
\end{array}
\end{array}
\right.  
\end{equation}
The study of the three systems above is carried out in an analogous manner, with the difference that the study of the last requires hypothesis \eqref{hyp const} since, for this system, the following coefficients $1-\sigma_{H}^{-}\,(1+\nu f_{H}^{-})$ and $1-\sigma_{H}^{+}\,(1+f_{H}^{+})$ are not necessarily positive. This is why we focus ourselves on the system governed by the hosts, that is \eqref{Pb trans en H}.

We will need the prerequisites mentioned in the section below.

\section{Prerequisites}

\subsection{Some properties on complex numbers}

In the sequel, we will need the following definition.

\begin{Def}
\label{defsector} Let $\omega \in \lbrack 0,\pi ]$, we denote 
\begin{equation*}
S_{\omega }=\left\{ 
\begin{array}{lll}
\left\{ z\in \CC:z\neq 0\text{ and }|\arg(z)|<\omega \right\} & \text{if} & 
\omega \in (0,\pi], \\ 
\ecart\,\RR_{+}\setminus \{0\} & \text{if} & \omega =0.
\end{array}
\right.
\end{equation*}
\end{Def}

\begin{Prop}
\label{Prop DoreLabbas} Let $z_{1},z_{2}\in \mathbb{C}\setminus \{0\}$. Then 
\begin{equation*}
|z_{1}+z_{2}|\geqslant \left( |z_{1}|+|z_{2}|\right) \left\vert \cos \left( 
\frac{\arg (z_{1})-\arg (z_{2})}{2}\right) \right\vert .
\end{equation*}
\end{Prop}

This result is proved in \cite{dore-labbas2011}, Proposition 4.9, p. 1879.

\begin{Prop}
\label{Prop DoreLabbas 2} Let $0<\alpha <\pi /2$ and $z\in S_{\alpha }$. We
have

\begin{enumerate}
\item $\left|\arg\left(1 - e^{-z}\right) - \arg\left(1 +
e^{-z}\right)\right| < \alpha$,

\item $\left|1 + e^{-z} \right| \geqslant 1 - e^{-\pi/(2 \tan(\alpha))}$,

\item $\dfrac{|z| \cos(\alpha)}{1 + |z|\cos(\alpha)} \leqslant |1 - e^{-z}|
\leqslant \dfrac{2|z|}{1 + |z| \cos(\alpha)}$.
\end{enumerate}
\end{Prop}

This result is proved in \cite{dore-labbas2011}, Proposition 4.10, p. 1880.

\subsection{Some properties on sectorial operators}

This subsection uses the definitions and results of \cite{haase2006}.

\begin{Def}
\label{Def op sect} Let $\omega \in (0,\pi)$. A linear
operator $A$ on a complex Banach space $E$ is said sectorial with angle $\omega$ if

\begin{enumerate}
\item $\sigma(A) \subset \overline{S_\omega}$,

\item $\dis M(A,\omega ^{\prime }):=\sup_{\lambda \in \CC\setminus \overline{%
S_{\omega ^{\prime }}}}\left\Vert \lambda (A-\lambda I)^{-1}\right\Vert _{\L %
(E)}<+\infty $ for every $\omega ^{\prime }\in (\omega ,\pi).$
\end{enumerate}

Here, $\sigma (A)$ denotes the spectrum of $A$. When $A$ verifies the two
points above, we write $A \in Sect(\omega )$.
\end{Def}

\begin{Prop}
\label{Prop sect} \hfill

\begin{enumerate}
\item If $(-\infty ,0) \subset \rho (A)$ and 
\begin{equation*}
M(A):=M(A,\pi ):=\underset{t>0}{\sup }\left\Vert t(A+tI)^{-1}\right\Vert
<+\infty ,
\end{equation*}%
then $M(A)\geqslant 1$ and 
\begin{equation*}
A\in Sect\left( \pi -\arcsin \left( \dfrac{1}{M(A)}\right) \right) .
\end{equation*}

\item If $A\in Sect(\omega _{A})$ and $\nu \in (0,1/2]$, then $A^{\nu }$
is well defined and $A^{\nu }\in Sect(\nu \omega _{A})$. We deduce that $-A^{\nu }$ generates an analytic semigroup.
\end{enumerate}
\end{Prop}

\subsection{A fundamental property on the $H^{\infty }$-calculus}

This property is based on the results of \cite{cowling1996}.

We set 
\begin{equation*}
H^{\infty }(S_{\omega })=\left\{ f:f\text{ is a bounded holomorphic function
on }S_{\omega }\right\} ,
\end{equation*}%
with $\omega \in (0,\pi)$. The property is stated as follows

\begin{Prop}
\label{Prop inv f} Let $A$ an injective sectorial operator with dense range.
If $f\in H^{\infty }(S_{\omega })$ is such that $1/f\in H^{\infty
}(S_{\omega })$ and 
\begin{equation*}
(1/f)(A)\in \mathcal{L}(E),
\end{equation*}%
then $f(A)$ is invertible with bounded inverse and 
\begin{equation}
\left[ f(A)\right] ^{-1}=(1/f)(A),  \label{HinfiniCalcul}
\end{equation}%
where $f(A)$ is defined by a Dunford's integral. For more details, see \cite{cowling1996} or \cite{dore-labbas2011}.
\end{Prop}

\subsection{Interpolation spaces}

We recall some properties on real interpolation spaces.

\begin{Def}
Let $T_{1}:D(T_{1})\subset E\longrightarrow E$ a linear operator such that 
\begin{equation}
(0,+\infty) \subset \rho (T_{1})\quad \text{and}\quad \exists
~C>0:\forall ~t>0,\quad \Vert t(T_{1}-tI)^{-1}\Vert _{\L (E)}\leqslant C.
\label{def T1}
\end{equation}%
Let $m\in \NN\setminus \{0\}$, $\theta \in (0,1)$ and $q\in [1,+\infty]$. We will use the following interpolation spaces characterized
in \cite{grisvard1969}: 
\begin{equation*}
(D(T_{1}^{m}),E)_{\theta ,q}=(E,D(T_{1}^{m}))_{1-\theta ,q}.
\end{equation*}
In particular, for $m=1$, we have the characterisation 
\begin{equation*}
(D(T_{1}),E)_{\theta ,q}:=\left\{ \psi \in E:t\longmapsto t^{1-\theta }\Vert
T_{1}(T_{1}-tI)^{-1}\psi \Vert _{E}\in L_{\ast }^{q}(0,+\infty ;\CC)\right\}
,
\end{equation*}%
where $L_{\ast }^{q}(0,+\infty ;\CC)$ is given by 
\begin{equation*}
L_{\ast }^{q}(0,+\infty ;\CC):=\left\{ f\in L^{q}(0,+\infty ):\left(
\int_{0}^{+\infty }|f(t)|^{q}\frac{dt}{t}\right) ^{1/q}<+\infty \right\}
,\quad \text{for }q\in \lbrack 1,+\infty \lbrack ,
\end{equation*}%
and for $q=+\infty $, by 
\begin{equation*}
L_{\ast }^{\infty }(0,+\infty ;\CC):=\left\{ f\text{ is measurable on }
(0,+\infty) : \underset{t\in (0,+\infty)}{\text{ess sup}} |f(t)|<+\infty \right\} .
\end{equation*}
We also define, for every $m\in \NN\setminus \{0\}$ 
\begin{equation*}
(D(T_{1}),E)_{m+\theta ,q}\;:=\;\left\{ \psi \in D(T_{1}^{m}):T_{1}^{m}\psi
\in (D(T_{1}),E)_{\theta ,q}\right\} ,
\end{equation*}%
and 
\begin{equation*}
(E,D(T_{1}))_{m+\theta ,q}\;:=\;\left\{ \psi \in D(T_{1}^{m}):T_{1}^{m}\psi
\in (E,D(T_{1}))_{\theta ,q}\right\} .
\end{equation*}
\end{Def}

\begin{Rem}
\label{Rem Réitération} We remark that $T_{1}^{m}$ is closed for every $%
m\in \NN\setminus \{0\}$ as $\rho (T_{1})\neq \emptyset $; on the other
hand, when $m\theta <1$, we have 
\begin{equation*}
(D(T_{1}^{m}),E)_{\theta ,q}=(E,D(T_{1}^{m}))_{1-\theta
,q}=(E,D(T_{1}))_{m-m\theta ,q}=(D(T_{1}),E)_{(m-1)+m\theta ,q}\subset
D(T_{1}^{m-1});
\end{equation*}%
see details in \cite{lunardi1996}, (2.1.13), p. 43, or \cite{grisvard1969},
p. 676, Theorem 6.
\end{Rem}
\begin{Rem} \label{remtriebel} 
Let $a,b\in \mathbb{R}$, $a<b$, $n\in \mathbb{N}\setminus \{0\}$ and $T_{2}$ the generator of a bounded analytic semigroup in $E$; we recall 
\begin{equation}
\left\{ 
\begin{array}{l}
x\longmapsto e^{(x-a)T_{2}}\psi \in L^{p}\left( a,b;E\right) \text{ and }%
x\longmapsto e^{(b-x)T_{2}}\psi \in L^{p}\left( a,b;E\right) \text{ for
every }\psi \in E,\medskip \\ 
x\longmapsto T_{2}^{n}e^{(x-a)T_{2}}\psi \in L^{p}\left( a,b;E\right)
\Longleftrightarrow \psi \in \left( D\left( T_{2}^{n}\right) ,E\right) _{%
\frac{1}{np},p},\medskip \\ 
x\longmapsto T_{2}^{n}e^{(b-x)T_{2}}\psi \in L^{p}\left( a,b;E\right)
\Longleftrightarrow \psi \in \left( D\left( T_{2}^{n}\right) ,E\right) _{%
\frac{1}{np},p},%
\end{array}%
\right.  \label{Reg triebel}
\end{equation}
where $p\in (1,+\infty)$; see Theorem p. 96 in \cite{triebel1978} for the two last statements.

Moreover, by the reiteration Theorem, it follows that, the three following properties are equivalent

\begin{enumerate}
\item $x\longmapsto e^{(x-a)T_2} \psi \in W^{n,p}(a,b;E)\cap
L^p\left(a,b;D(T_2^n)\right)$,

\item $x\longmapsto e^{(b-x)T_2} \psi \in W^{n,p}(a,b;E)\cap
L^p\left(a,b;D(T_2^n)\right)$,

\item $\psi \in \left(D(T_2),E\right)_{n-1+\frac{1}{p},p}$.
\end{enumerate}
\end{Rem}

\subsection{Recall on the Dore-Venni Theorem}

We now recall the famous Theorem of Dore and Venni, see \cite{dore-venni1987}%
, where the complex Banach space $E$ is supposed to be UMD, see \cite{bourgain1983}
and \cite{burkholder1981}.

\begin{Def}
Let $\theta \in [0,\pi)$. We denote by BIP$(E,\theta )$, the
class of injective sectorial operators $-T_{3}$ such that

\begin{itemize}
\item[$i)$] $\overline{D(T_3)} = \overline{R(T_3)} = E,$

\item[$ii)$] $\forall~ s \in \RR, \quad (-T_3)^{is} \in \L (E),$

\item[$iii)$] $\exists~ C \geq 1 ,~ \forall~ s \in \RR, \quad
\|(-T_3)^{is}\|_{\L (E)} \leq C e^{|s|\theta}$,
\end{itemize}

see~\cite{pruss-sohr1990}, p. 430.
\end{Def}

\begin{Th}
\label{Th DorVen} Let $-T_{3}\in $ BIP\thinspace $(E,\theta )$ with $\theta
\in (0,\pi/2)$ and \mbox{$g\in L^p(a,b;E)$}, where $p\in (1,+\infty)$ and $a,b\in \mathbb{R}$ with $a<b$. Then, for almost every $x\in
(a,b)$, we have 
\begin{equation*}
\int_{a}^{x}e^{(x-s)T_{3}}g(s)\,ds\in D(T_{3})\quad \text{and}\quad
\int_{x}^{b}e^{(s-x)T_{3}}g(s)\,ds\in D(T_{3}).
\end{equation*}
Moreover
\begin{equation*}
x\longmapsto T_{3}\int_{a}^{x}e^{(x-s)T_{3}}g(s)\,ds\in L^{p}(a,b;E)\quad 
\text{and}\quad x\longmapsto T_{3}\int_{x}^{b}e^{(s-x)T_{3}}g(s)\,ds\in
L^{p}(a,b;E).
\end{equation*}
\end{Th}

\section{Operational formulation of system \eqref{Pb trans en A}}

We consider, in the UMD Banach space $X=L^{p}(0,1)$, where $p\in
(1,+\infty)$, linear operators $Q$, $Q^{-}$ and $Q^{+}$ defined by 
\begin{equation*}
\left\{ 
\begin{array}{lll}
D(Q) & = & \{\phi \in W^{2,p}(0,1):\phi (0)=\phi (1)=0\} \\ 
\ecart(Q\phi )(y) & = & \phi ^{\prime \prime }(y),
\end{array}%
\right.
\end{equation*}
\begin{equation*}
\left\{ 
\begin{array}{lll}
D(Q^{-}) & = & \{\phi \in W^{2,p}(0,1):\phi (0)=\phi (1)=0\} \\ 
\ecart(Q^{-}\phi )(y) & = & \phi''(y) - \dfrac{1-\sigma_{H}^{-}(1+\nu f_H^-)
}{d_{H}^{-}}\phi (y),
\end{array}
\right.
\end{equation*}
and
\begin{equation*}
\left\{ 
\begin{array}{lll}
D(Q^{+}) & = & \{\phi \in W^{2,p}(0,1):\phi (0)=\phi (1)=0\} \\ 
\ecart(Q^{+}\phi )(y) & = & \phi''(y) - \dfrac{1-\sigma_{H}^{+}(1+f_H^+)}{d_{H}^{+}}\phi (y).
\end{array}
\right.
\end{equation*}

\begin{Prop}
\label{Prop Qpm} Assume that \eqref{hyp const} holds, then operators $Q$, $Q^{-}$ and $Q^{+}$ are linear closed with dense domains in $X$ and verify 
\begin{equation}
\left\{ 
\begin{array}{l}
\forall \,\eta \in (0,\pi),~S_{\pi -\eta }\cup \{0\}\subset \rho
(Q) \\ 
\ecart

\exists \,C>0,~\forall \,z\in S_{\pi -\eta }\cup \{0\},\quad\Vert
(Q-zI)^{-1}\Vert _{\L (X)}\leqslant \dfrac{C}{1+|z|},%
\end{array}%
\right.  \label{estimation Q}
\end{equation}%
and 
\begin{equation}
\left\{ 
\begin{array}{l}
\forall \,\eta \in (0,\pi),~S_{\pi -\eta }\cup \{0\}\subset \rho(Q^{\pm }) \\ 
\ecart

\exists \,C>0,~\forall \,z\in S_{\pi -\eta }\cup \{0\}, \quad \Vert (Q^{\pm}-zI)^{-1}\Vert _{\L (X)}\leqslant \dfrac{C}{1+|z|}.
\end{array}
\right.  \label{estimation Q_i}
\end{equation}
Moreover, there exists an open ball $B(0,\delta )$, $\delta >0$, such that $\overline{B(0,\delta )}\subset \rho (Q)$, $\overline{B(0,\delta )}\subset \rho (Q^{\pm })$ and the estimates \eqref{estimation Q} and \eqref{estimation Q_i} remain true in $S_{\pi-\eta }\cup \overline{B(0,\delta )}$. Here, $\rho (Q)$ and $\rho (Q^{\pm })$
denote respectively the resolvent sets of $Q$ and $Q^{\pm }$.
\end{Prop}

\begin{proof}
The proof of \eqref{estimation Q} is well known, while the one of \eqref{estimation Q_i} follows essentially from \eqref{hyp const} and the fact that the second derivative operator with Dirichlet boundary conditions admits for eigenvalues the sequence $-k^2\pi^2$, $k\in \NN^*$.
\end{proof}

\begin{Rem}
The definition of sectorial operator $A$ has been given in \refD{Def op sect}. We recall that this notion corresponds, in the Hilbert case, to the positivity of the scalar product $<Au,u>$. A typical example of such an operator A is an elliptic operator.

Here, $-Q$, $-Q^-$ and $-Q^+$ are clearly sectorial and satisfy \refD{Def op sect}.
\end{Rem}

Using the usual notation for vector-valued functions, we set 
\begin{equation*}
\left\{ 
\begin{array}{lll}
h_{I}(x)(y) & := & h_{I}(x,y) \\ 
h_{S}(x)(y) & := & h_{S}(x,y)
\end{array}%
\right. \quad \text{and}\quad \left\{ 
\begin{array}{lll}
g_{I}(x)(y) & := & g_{I}(x,y)=\dfrac{n_{I}}{d_{H}^{-}}(x,y) \\ 
\ecart g_{S}(x)(y) & := & g_{S}(x,y)=\dfrac{n_{S}}{d_{H}^{+}}(x,y).
\end{array}
\right.
\end{equation*}
Thus, system \eqref{Pb trans en A} can be written as 
\begin{equation}
\left\{ 
\begin{array}{l}
\begin{array}{llll}
h_{I}''(x) + Q^- h_{I}(x)-\dfrac{\lambda}{d_{H}^{-}}h_{I}(x) & = & 
g_{I}(x), & \text{a.e.~}x\in (-\ell ,0) \\ \ecart 
h_{S}''(x) + Q^+ h_{S}(x)-\dfrac{\lambda }{d_{H}^{+}}h_{S}(x) & 
= & g_{S}(x), & \text{a.e.~}x\in (0,L)
\end{array}
\\ 
\ecart
\begin{array}{lll}
h_{I}(-\ell ) & = & h_{S}(L)=0 \\ 
h_{I}(0) & = & h_{S}(0) \\ 
\beta_{I} h'_{I}(0) & = & \beta_{S} h'_{S}(0),%
\end{array}%
\end{array}%
\right.  \label{syst w eda}
\end{equation}%
where 
\begin{equation*}
g_{I}\in L^{p}(-\ell ,0;X)=L^{p}(-\ell ,0;L^{p}(0,1))=L^{p}(\Omega _{I}),
\end{equation*}%
and 
\begin{equation*}
g_{S}\in L^{p}(0,L;X)=L^{p}(0,L;L^{p}(0,1))=L^{p}(\Omega _{S}).
\end{equation*}

\section{Spectral properties of operators $Q^{\pm }$}

Now, we have to specify the sector of the spectral parameter $\lambda$. In all the sequel 
\begin{equation}
\lambda \in S_{\pi -\varepsilon },  \label{lambda dans S}
\end{equation}%
where $\varepsilon \in \left(0, \pi/2\right)$ is fixed and will be specified later. We set 
\begin{equation*}
Q_{\lambda }^{-} = Q - \dfrac{1-\sigma_H^-(1+\nu f_H^-)}{d_H^-}I - \dfrac{\lambda }{d_{H}^{-}} I= Q^{-}-\dfrac{\lambda }{d_{H}^{-}}I\quad \text{and}\quad
Q_{\lambda }^{+} = Q - \dfrac{1-\sigma_H^+ (1+f_H^+)}{d_H^+} I - \dfrac{\lambda }{d_{H}^{+}}I = Q^{+} -\dfrac{\lambda }{d_{H}^{+}}I;
\end{equation*}
thus, we have 
\begin{equation*}
D(Q_{\lambda }^{-})=D(Q^{-})=D(Q^{+})=D(Q_{\lambda }^{+}).
\end{equation*}
According to \refP{Prop Qpm}, operators $-Q^{\pm }$ are sectorial in $X$. The same is true for operators $-Q_{\lambda }^{\pm }$. Indeed, we have $(-\infty ,0]\subset \rho (-Q_{\lambda }^{\pm })$. We set 
\begin{equation*}
M(-Q_{\lambda }^{\pm }):=M(-Q_{\lambda }^{\pm },\pi ):=\sup_{t>0}\Vert
t(-Q_{\lambda }^{\pm }+tI)^{-1}\Vert _{\L (X)}.
\end{equation*}%
Due to \refP{Prop sect}, for all $\lambda \in S_{\pi -\varepsilon}$, we obtain
\begin{equation*}
M(-Q_{\lambda }^{\pm })\leqslant \sup_{t>0}\left( \frac{t}{\cos \left( \frac{1}{2}\arg \left( \dfrac{\lambda }{d_{H}^{\pm }}+t\right) \right) }\,\dfrac{1}{\left\vert \dfrac{\lambda }{d_{H}^{\pm }}+t\right\vert }\right) .
\end{equation*}%
Two cases are possible:

\begin{enumerate}
\item If $|\arg (\lambda )|<\pi /2$, then: 
\begin{equation*}
\forall \,t>0,\quad \left\vert \dfrac{\lambda }{d_{H}^{\pm }}+t\right\vert
\geqslant t.
\end{equation*}

\item If $\pi /2\leqslant |\arg (\lambda )|<\pi -\varepsilon $, then 
\begin{equation*}
\forall \,t>0,\quad \left\vert \dfrac{\lambda }{d_{H}^{\pm }}+t\right\vert
\geqslant t\sin (\alpha ) \geqslant t\sin (\varepsilon ),
\end{equation*}
since $\alpha \in (\varepsilon ,\pi /2]$. 
\end{enumerate}
Thus, in these two cases, there exists a constant $C>0$ independent of $\lambda $ such that 
\begin{equation*}
M(-Q_{\lambda }^{\pm })\leqslant \frac{C}{\cos \left( \dfrac{\pi
-\varepsilon }{2}\right) }=\frac{C}{\sin (\varepsilon )}<+\infty ,
\end{equation*}%
so 
\begin{equation*}
-Q_{\lambda }^{\pm }\in Sect\left( \pi -\arcsin \left( \dfrac{1}{%
M(-Q_{\lambda }^{\pm })}\right) \right) .
\end{equation*}%
We deduce that the two following operators 
\begin{equation*}
P_{\lambda }^{-}=-\left( -\left( Q^{-}-\dfrac{\lambda }{d_{H}^{-}}I\right)
\right) ^{1/2}\quad \text{and}\quad P_{\lambda }^{+}=-\left( -\left( Q^{+}-
\dfrac{\lambda }{d_{H}^{+}}I\right) \right) ^{1/2},
\end{equation*}
are well defined and have the same domain
\begin{equation*}
D(P_{\lambda }^{-})=D(P_{\lambda }^{+})=D\left( \left( -Q^{-}\right)
^{1/2}\right) =D\left( \left( -Q^{+}\right) ^{1/2}\right) .
\end{equation*}%
It is well known that operators $-\left( -Q^{-}\right) ^{1/2}$ and $-\left(
-Q^{+}\right) ^{1/2}$ generate analytic semigroups in $X$, see \cite{balakrishnan1960}.

Using Lemma 4.2 in \cite{flmt2021} as well as estimates (28) and (29) in 
\cite{flmt2021}, there exist $\varepsilon ^{\pm }>0$ and $C^{\pm }>0$ independent
of $\lambda $ such that 
\begin{equation}\label{estimation de la resolvante de P pm}
\left\{ 
\begin{array}{l}
\forall \,z\in \left\{ z\in \CC\setminus \{0\}:|\arg (z)|\leqslant \dfrac{\pi}{2} +\varepsilon^{\pm }\right\} \\ \ecart
\left\Vert (P_{\lambda }^{\pm }-zI)^{-1}\right\Vert _{\L (X)} \leqslant \dfrac{C^{\pm }}{\sqrt{1+|\lambda |}+|z|}.
\end{array}
\right.  
\end{equation}

\section{Resolution of system \eqref{syst w eda}}

From this section, we will assume in the sequel that
\begin{equation*}
\lambda \in S_{2(\pi -\varepsilon)/3},
\end{equation*}
where $\varepsilon \in \left(0,\pi/2\right)$ is fixed small enough; this will allow us to apply the $H^{\infty }$-calculus for the resolution of system \eqref{syst w eda}.

\subsection{Calculus of the determinant operator}
\label{Sect Calcul Det}
Using operators $P_{\lambda }^{-}$ and $P_{\lambda}^{+}$, system \eqref{syst w eda} writes as follows 
\begin{equation}
\left\{ 
\begin{array}{l}
\begin{array}{llll}
h_{I}''(x)-(P_{\lambda }^{-})^{2}h_{I}(x) & = & g_{I}(x), & 
\text{a.e.~}x\in (-\ell ,0) \\ \ecart 
h''_{S}(x)-(P_{\lambda }^{+})^{2}h_{S}(x) & = & g_{S}(x), & \text{a.e.~}x\in (0,L)
\end{array}
\\ \ecart
\begin{array}{lll}
h_{I}(-\ell ) & = & h_{S}(L)=0 \\ 
h_{I}(0) & = & h_{S}(0) \\ 
\beta _{I}h'_{I}(0) & = & \beta_{S}h'_{S}(0).
\end{array}
\end{array}
\right.  \label{syst w eda P}
\end{equation}
Applying a similar method as in \cite{menad2019}, the solution $(h_{I},h_{S})$ of system \eqref{syst w eda P} writes as 
\begin{equation*}
\left\{ 
\begin{array}{llll}
h_{I}(x) & = & e^{(x+\ell )P_{\lambda }^{-}}\gamma _{I}+e^{-xP_{\lambda
}^{-}}\delta _{I}+w_{I}(g_{I})(x), & x\in (-\ell ,0) \\ 
\ecart h_{S}(x) & = & e^{xP_{\lambda }^{+}}\gamma _{S}+e^{(L-x) P_{\lambda}^{+}}\delta _{S}+w_{S}(g_{S})(x), & x\in (0,L),
\end{array}
\right.
\end{equation*}%
where 
\begin{equation*}
\left\{ 
\begin{array}{llll}
w_{I}(g_{I})(x) & = & \dis\frac{1}{2}\int_{-\ell }^{x}e^{(x-t)P_{\lambda
}^{-}}(P_{\lambda }^{-})^{-1}g_{I}(t)~dt+\frac{1}{2}\int_{x}^{0}e^{(t-x)P_{\lambda }^{-}}(P_{\lambda }^{-})^{-1}g_{I}(t)~dt, & x\in (-\ell ,0) \\ 
\ecart w_{S}(g_{S})(x) & = & \dis\frac{1}{2}\int_{0}^{x}e^{(x-t) P_{\lambda}^{+}}(P_{\lambda }^{+})^{-1}g_{S}(t)~dt+\frac{1}{2}\int_{x}^{L}e^{(t-x)P_{\lambda }^{+}}(P_{\lambda }^{+})^{-1}g_{S}(t)~dt, & x\in (0,L).
\end{array}
\right.
\end{equation*}%
It is easy to see that all these integrals are well defined due to the semigroups properties. We will prove later that functions $h_{I}$ and $h_{S}$ satisfy the following optimal regularity 
\begin{equation*}
\left\{ 
\begin{array}{l}
h_{I}\in W^{2,p}(-\ell ,0;X)\cap L^{p}\left( -\ell ,0;D\left( (P_{\lambda
}^{-})^{2}\right) \right) \\ \ecart 
h_{S}\in W^{2,p}(0,L;X)\cap L^{p}\left( 0,L;D\left( (P_{\lambda
}^{+})^{2}\right) \right).
\end{array}
\right.
\end{equation*}
Now, we have to find the constants $\gamma _{I}$, $\gamma _{S}$, $\delta _{I}$
and $\delta _{S}$. Thanks to the boundary conditions, we have
\begin{equation*}
\left\{ 
\begin{array}{lllll}
0 & = & h_{I}(-\ell ) & = & \gamma _{I}+e^{\ell P_{\lambda }^{-}}\delta
_{I}+w_{I}(g_{I})(-\ell ) \\ \ecart
0 & = & h_{S}(L) & = & e^{LP_{\lambda }^{+}}\gamma _{S}+\delta
_{S}+w_{S}(g_{S})(L),
\end{array}
\right. 
\end{equation*}
where 
\begin{equation*}
\left\{ 
\begin{array}{lll}
w_{I}(g_{I})(-\ell ) & = & \dis\frac{1}{2}\int_{-\ell }^{0}e^{(t+\ell
)P_{\lambda }^{-}}(P_{\lambda }^{-})^{-1}g_{I}(t)~dt \\ 
\ecart w_{S}(g_{S})(L) & = & \dis\frac{1}{2}\int_{0}^{L}e^{(L-t)P_{\lambda
}^{+}}(P_{\lambda }^{+})^{-1}g_{S}(t)~dt.%
\end{array}%
\right. 
\end{equation*}
Hence
\begin{equation}
\left\{ 
\begin{array}{lllll}
\gamma _{I} & = & -e^{\ell P_{\lambda }^{-}}\delta _{I}-w_{I}(g_{I})(-\ell )
&  &  \\ 
\ecart\delta _{S} & = & -e^{LP_{\lambda }^{+}}\gamma _{S}-w_{S}(g_{S})(L). & 
& 
\end{array}%
\right.   \label{gamma I et delta S}
\end{equation}%
On the other hand, we have 
\begin{equation*}
\left\{ 
\begin{array}{llll}
h_{I}^{\prime }(x) & = & P_{\lambda }^{-}e^{(x+\ell )P_{\lambda }^{-}}\gamma
_{I}-P_{\lambda }^{-}e^{-xP_{\lambda }^{-}}\delta _{I}+w_{I}^{\prime
}(g_{I})(x), & x\in (-\ell ,0) \\ \ecart 
h_{S}^{\prime }(x) & = & P_{\lambda }^{+}e^{xP_{\lambda }^{+}}\gamma_{S}-P_{\lambda }^{+}e^{(L-x)P_{\lambda }^{+}}\delta _{S}+w_{S}^{\prime}(g_{S})(x), & x\in (0,L),
\end{array}
\right. 
\end{equation*}%
with 
\begin{equation*}
\left\{ 
\begin{array}{llll}
w_{I}^{\prime }(g_{I})(x) & = & \dis\frac{1}{2}\int_{-\ell}^{x}e^{(x-t)P_{\lambda }^{-}}g_{I}(t)~dt-\frac{1}{2}\int_{x}^{0}e^{(t-x)P_{\lambda }^{-}}g_{I}(t)~dt, & x\in (-\ell ,0) \\ 
\ecart w_{S}^{\prime }(g_{S})(x) & = & \dis\frac{1}{2}
\int_{0}^{x}e^{(x-t)P_{\lambda }^{+}}g_{S}(t)~dt-\frac{1}{2}
\int_{x}^{L}e^{(t-x)P_{\lambda }^{+}}g_{S}(t)~dt, & x\in (0,L).
\end{array}
\right. 
\end{equation*}
From the transmission conditions, we formally deduce that 
\begin{equation}
\left\{ 
\begin{array}{lll}
e^{\ell P_{\lambda }^{-}}\gamma _{I}+\delta _{I}+w_{I}(g_{I})(0) & = & 
\gamma _{S}+e^{LP_{\lambda }^{+}}\delta _{S}+w_{S}(g_{S})(0) \\ \ecart
\beta _{I}\left( P_{\lambda }^{-}e^{\ell P_{\lambda }^{-}}\gamma_{I}-P_{\lambda }^{-}\delta _{I}+w_{I}^{\prime }(g_{I})(0)\right)  & = & 
\beta _{S}\left( P_{\lambda }^{+}\gamma _{S}-P_{\lambda }^{+}e^{LP_{\lambda}^{+}}\delta _{S}+w_{S}^{\prime }(g_{S})(0)\right),
\end{array}
\right.   \label{syst det CT}
\end{equation}%
where 
\begin{equation*}
\left\{ 
\begin{array}{llll}
w_{I}(g_{I})(0) & = & \dis\frac{1}{2}\int_{-\ell }^{0}e^{-tP_{\lambda
}^{-}}(P_{\lambda }^{-})^{-1}g_{I}(t)~dt &  \\ 
\ecart w_{S}(g_{S})(0) & = & \dis\frac{1}{2}\int_{0}^{L}e^{tP_{\lambda
}^{+}}(P_{\lambda }^{+})^{-1}g_{S}(t)~dt, & 
\end{array}%
\right. 
\end{equation*}%
and 
\begin{equation}
\left\{ 
\begin{array}{lllll}
w_{I}^{\prime }(g_{I})(0) & = & \dis\frac{1}{2}\int_{-\ell
}^{0}e^{-tP_{\lambda }^{-}}g_{I}(t)~dt & = & P_{\lambda }^{-}w_{I}(g_{I})(0)
\\ 
\ecart w_{S}^{\prime }(g_{S})(0) & = & \dis-\frac{1}{2}\int_{0}^{L}e^{tP_{%
\lambda }^{+}}g_{S}(t)~dt & = & -P_{\lambda }^{+}w_{S}(g_{S})(0).%
\end{array}%
\right.   \label{w' = w}
\end{equation}%
We will see, later, that the equalities in \eqref{syst det CT} are well defined since $\delta _{I}\in D(P_{\lambda }^{-})$ and $\gamma _{S}\in D(P_{\lambda }^{+})$. On the other hand, we know that integrals in \eqref{w' = w} belong respectively to $D(P_{\lambda }^{-})$ and $D(P_{\lambda }^{+})$ from Proposition~1.2, (ii), p. 20 in \cite{sinestrari1985}. 

Applying $(P_{\lambda }^{-})^{-1}$ on the second line
of system \eqref{syst det CT}, we obtain 
\begin{equation*}
\left\{ 
\begin{array}{lll}
e^{\ell P_{\lambda }^{-}}\gamma _{I}+\delta _{I}+w_{I}(g_{I})(0) & = & 
\gamma _{S}+e^{LP_{\lambda }^{+}}\delta _{S}+w_{S}(g_{S})(0) \\ \\
\beta_{I}\left( e^{\ell P_{\lambda }^{-}}\gamma _{I}-\delta
_{I}+(P_{\lambda }^{-})^{-1}w_{I}^{\prime }(g_{I})(0)\right)  & = & \beta
_{S}\left( (P_{\lambda }^{-})^{-1}P_{\lambda }^{+}\gamma _{S}-(P_{\lambda
}^{-})^{-1}P_{\lambda }^{+}e^{LP_{\lambda }^{+}}\delta _{S}\right) \\ \ecart
&&\dis  + \beta_S (P_{\lambda}^{-})^{-1}w_{S}^{\prime }(g_{S})(0),
\end{array}%
\right. 
\end{equation*}%
and, since $(P_{\lambda }^{-})^{-1}$ and $P_{\lambda }^{+}$ commute on the
domain of $P_{\lambda }^{+}$, we have 
\begin{equation*}
\left\{ 
\begin{array}{lll}
e^{\ell P_{\lambda }^{-}}\gamma _{I}+\delta _{I}+w_{I}(g_{I})(0) & = & 
\gamma _{S}+e^{LP_{\lambda }^{+}}\delta _{S}+w_{S}(g_{S})(0) \\ \\
\beta _{I}\left( e^{\ell P_{\lambda }^{-}}\gamma _{I}-\delta_{I}+(P_{\lambda }^{-})^{-1}w_{I}^{\prime }(g_{I})(0)\right)  & = & \beta_{S}\left( P_{\lambda }^{+}(P_{\lambda }^{-})^{-1}\gamma _{S}-P_{\lambda}^{+}(P_{\lambda }^{-})^{-1}e^{LP_{\lambda }^{+}}\delta _{S}\right) \\ \ecart
&& \dis + \beta_S (P_{\lambda}^{-})^{-1}w_{S}^{\prime }(g_{S})(0).
\end{array}
\right. 
\end{equation*}%
Using \eqref{w' = w}, we deduce that
\begin{equation}
\left\{ 
\begin{array}{lll}
e^{\ell P_{\lambda }^{-}}\gamma _{I}+\delta _{I}+w_{I}(g_{I})(0) & = & 
\gamma _{S}+e^{LP_{\lambda }^{+}}\delta _{S}+w_{S}(g_{S})(0) \\ \ecart
\beta _{I}\left( e^{\ell P_{\lambda }^{-}}\gamma _{I}-\delta
_{I}+w_{I}(g_{I})(0)\right)  & = & \beta _{S}\,P_{\lambda }^{+}(P_{\lambda
}^{-})^{-1}\left( \gamma _{S}-e^{LP_{\lambda }^{+}}\delta
_{S}-w_{S}(g_{S})(0)\right).
\end{array}%
\right.   \label{syst det CT 2}
\end{equation}%
We set
\begin{equation*}
R_{I}=w_{I}(g_{I})(0)-e^{\ell P_{\lambda }^{-}}w_{I}(g_{I})(-\ell )=\frac{1}{2} \int_{-\ell }^{0}e^{-tP_{\lambda }^{-}}\left( I-e^{2(t+\ell )P_{\lambda}^{-}}\right) (P_{\lambda }^{-})^{-1}g_{I}(t)~dt,
\end{equation*}%
and 
\begin{equation*}
R_{S}=w_{S}(g_{S})(0)-e^{LP_{\lambda }^{+}}w_{S}(g_{S})(L)=\frac{1}{2}%
\int_{0}^{L}e^{tP_{\lambda }^{+}}\left( I-e^{2(L-t)P_{\lambda }^{+}}\right)
(P_{\lambda }^{+})^{-1}g_{S}(t)~dt.
\end{equation*}%
Using \eqref{gamma I et delta S}, system \eqref{syst det CT 2} becomes 
\begin{equation*}
\left\{ 
\begin{array}{lll}
\left( I-e^{2\ell P_{\lambda }^{-}}\right) \delta _{I}+R_{I} & = & \left(
I-e^{2LP_{\lambda }^{+}}\right) \gamma _{S}+R_{S} \\ \ecart
\beta_{I}\left( -\left( I+e^{2\ell P_{\lambda }^{-}}\right) \delta_{I}+R_{I}\right) & = & \beta_{S}\,P_{\lambda }^{+}(P_{\lambda}^{-})^{-1}\left( \left( I+e^{2LP_{\lambda }^{+}}\right) \gamma_{S}-R_{S}\right),
\end{array}
\right. 
\end{equation*}
hence 
\begin{equation}
\left\{ 
\begin{array}{lll}
\left( I-e^{2\ell P_{\lambda }^{-}}\right) \delta _{I}-\left(
I-e^{2LP_{\lambda }^{+}}\right) \gamma _{S} & = & R_{S}-R_{I} \\ 
\ecart
\beta_{I}\,\left( I+e^{2\ell P_{\lambda }^{-}}\right) \delta
_{I} + \beta _{S}\,P_{\lambda }^{+}(P_{\lambda }^{-})^{-1}\left(I+e^{2LP_{\lambda }^{+}}\right) \gamma _{S} & = & \beta _{I}\,R_{I} + \beta
_{S}\,P_{\lambda }^{+}(P_{\lambda }^{-})^{-1}R_{S}.
\end{array}%
\right.   \label{syst final}
\end{equation}
Note that all the coefficient operators of system \eqref{syst final} are bounded. Therefore, the abstract determinant of this system is given by 
\begin{equation*}
\begin{array}{lll}
D_{\lambda } & = & \beta_{I} \left( I+e^{2\ell P_{\lambda }^{-}}\right)
\left( I-e^{2LP_{\lambda }^{+}}\right) + \beta _{S}\,P_{\lambda }^{+}\left(
P_{\lambda }^{-}\right) ^{-1}\left( I-e^{2\ell P_{\lambda }^{-}}\right)
\left( I+e^{2LP_{\lambda }^{+}}\right)  \\ \ecart 
& = & \beta_{I}\left( I+e^{2\ell P_{\lambda }^{-}}\right) \left(I-e^{2LP_{\lambda }^{+}}\right) \Pi _{\lambda },
\end{array}%
\end{equation*}
where 
\begin{equation*}
\Pi _{\lambda }=I+\frac{\beta _{S}}{\beta _{I}}\,P_{\lambda
}^{+}(P_{\lambda }^{-})^{-1}\left( I-e^{2\ell P_{\lambda }^{-}}\right)
\left( I+e^{2LP_{\lambda }^{+}}\right) \left( I+e^{2\ell P_{\lambda
}^{-}}\right) ^{-1}\left( I-e^{2LP_{\lambda }^{+}}\right) ^{-1}.
\end{equation*}
We used the fact that operators $\left( I-e^{2LP_{\lambda }^{+}}\right)$ and $\left( I+e^{2\ell P_{\lambda}^{-}}\right)$ are boundedly invertible respectively in virtue of Proposition 2.3.6, p. 60, in \cite{lunardi1996} and Lemma 5.2, p. 1883 in \cite{dore-labbas2011}.

\subsection{Inversion of the determinant}

We recall that $\lambda \in S_{2(\pi -\varepsilon)/3}$ and 
\begin{equation*}
P_{\lambda }^{-}=-\left( -\left( Q-\frac{1-\sigma _{H}^{-}(1+\nu f_H^-)}{d_{H}^{-}}I - \frac{\lambda }{d_{H}^{-}}I\right) \right) ^{1/2}\quad \text{and}\quad P_{\lambda}^{+}=-\left( -\left( Q-\frac{1-\sigma _{H}^{+}(1+f_{H}^+)}{d_{H}^{+}}I - \frac{\lambda }{d_{H}^{+}}I\right) \right) ^{1/2}.
\end{equation*}
We set
\begin{equation*}
\lambda _{-}=-\frac{1-\sigma _{H}^{-}(1+\nu f_H^-)}{d_{H}^{-}}-\frac{\lambda }{d_{H}^{-}} \quad \text{and}\quad \lambda _{+}=-\frac{1-\sigma _{H}^{+}(1+f_{H}^+)}{d_{H}^{+}}-\frac{\lambda }{d_{H}^{+}}.
\end{equation*}
Thus 
\begin{equation*}
P_{\lambda }^{-}=-\left( -Q-\lambda _{-}I\right) ^{1/2}\quad \text{and}\quad
P_{\lambda }^{+}=-\left( -Q-\lambda _{+}I\right) ^{1/2}.
\end{equation*}%
For all $z\in \pi ^{2}+S_{(\pi -\varepsilon)/3}$, we define the following function
\begin{equation}
f_{\lambda }(z)=1+\frac{\beta_{S}}{\beta _{I}}\frac{\sqrt{z-\lambda _{+}}%
}{\sqrt{z-\lambda _{-}}}\frac{\left( 1-e^{-2\ell \sqrt{z-\lambda _{-}}%
}\right) \left( 1+e^{-2L\sqrt{z-\lambda _{+}}}\right) }{\left( 1+e^{-2\ell 
\sqrt{z-\lambda _{-}}}\right) \left( 1-e^{-2L\sqrt{z-\lambda _{+}}}\right) },
\label{f_lambda (z)}
\end{equation}
in order to have 
\begin{equation*}
f_{\lambda }(-Q)=\Pi _{\lambda }.
\end{equation*}
We will analyze the following complex quantities
\begin{equation*}
Z_{+}=\sqrt{z-\lambda _{+}}\quad \text{and}\quad Z_{-}=\sqrt{z-\lambda _{-}},
\end{equation*}
as well as 
\begin{equation*}
\left( 1\pm e^{-2L\sqrt{z-\lambda _{+}}}\right) \quad \text{and}\quad \left(
1\pm e^{-2\ell \sqrt{z-\lambda _{-}}}\right) 
\end{equation*}
for 
\begin{equation*}
z\in \pi ^{2}+S_{(\pi -\varepsilon)/3}\quad \text{and}\quad \lambda \in
S_{2(\pi -\varepsilon)/3}.
\end{equation*}
We will then need the following essential technical lemma in order to apply \refP{Prop DoreLabbas 2}. Then 
\begin{equation}
Z_{+}=\sqrt{z-\lambda _{+}}\in S_{(\pi -\varepsilon)/3}\subset S_{\pi
/3}\quad \text{and}\quad Z_{-}=\sqrt{z-\lambda _{-}}\in S_{(\pi -\varepsilon)/3}\subset S_{\pi /3},
\end{equation}
thus
\begin{equation*}
2LZ_{+}=2L\sqrt{z-\lambda _{+}}\in S_{(\pi -\varepsilon)/3}\subset S_{\pi
/3}\quad \text{and}\quad 2\ell Z_{-}=2\ell\sqrt{z-\lambda _{-}}\in S_{(\pi
-\varepsilon)/3}\subset S_{\pi /3}.
\end{equation*}

\begin{Lem}\label{Lem estim Zpm} 
Let $z\in \pi ^{2}+S_{(\pi -\varepsilon)/3}$ and $\lambda \in S_{2(\pi -\varepsilon)/3}$, with $0 < \varepsilon < \pi/2$ fixed. We have

\begin{enumerate}
\item $\dis |Z_\pm| \geqslant \sqrt{\dfrac{\sqrt{3}\, (\pi^2-r_0)}{2} \,
\sin(\varepsilon/2)}$,

\item $\left\{
\begin{array}{lll}
\left|\arg\left(1 - e^{-2L Z_+}\right) - \arg\left(1 +
e^{-2LZ_+}\right)\right| & < & \dfrac{\pi}{3}-\dfrac{\varepsilon}{3} \\ 
\ecart \left|\arg\left(1 - e^{-2\ell Z_-}\right) - \arg\left(1 + e^{-2\ell Z_-}\right)\right| & < & \dfrac{\pi}{3}-\dfrac{\varepsilon}{3},
\end{array}
\right.$

\item $\left\{
\begin{array}{lllll}
\dis\left|1 + e^{-2LZ_+} \right| & \geqslant & \dis 1 - e^{-\pi/(2
\tan((\pi-\varepsilon)/3))} & \geqslant & \dis 1 - e^{-\pi/2\sqrt{3}} \\ 
\ecart \dis \left|1 + e^{-2\ell Z_-} \right| & \geqslant & \dis 1 -
e^{-\pi/(2 \tan((\pi-\varepsilon)/3))} & \geqslant & \dis 1 - e^{-\pi/2\sqrt{3}},
\end{array}
\right.$

\item $\left\{
\begin{array}{lllll}
\dis \dfrac{L|Z_+|}{1 + L|Z_+|} & \leqslant & \dis |1 - e^{-2LZ_+}| & 
\leqslant & \dis \dfrac{4L|Z_+|}{1 + L|Z_+|} \\ 
\ecart \dis \dfrac{\ell|Z_-|}{1 + \ell|Z_-|} & \leqslant & \dis |1 -
e^{-2\ell Z_-}| & \leqslant & \dis \dfrac{4\ell|Z_-|}{1 + \ell|Z_-|}.
\end{array}
\right.$
\end{enumerate}
\end{Lem}

\begin{proof}
Statements 2., 3., and 4. are proved by a direct application of \refP{Prop DoreLabbas 2}. 

Let us show statement 1. Since $z \in \pi^2 + S_{(\pi - \varepsilon)/3}$, there exists $z' \in S_{(\pi - \varepsilon)/3}$ such that 
$$z = z' + \pi^2.$$ 
We set
$$z_- = z + \frac{1-\sigma_H^-(1+\nu f_H^-)}{d_H^-} \quad \text{and} \quad z_+ = z + \frac{1-\sigma_H^+(1+f_H^+)}{d_H^+},$$
thus
$$z-\lambda_- = z_- + \frac{\lambda}{d_H^-} \quad \text{and} \quad z - \lambda_+ = z_+ + \frac{\lambda}{d_H^+}.$$
If 
$$\frac{1-\sigma_H^-(1+\nu f_H^-)}{d_H^-} < 0,$$
then
$$|z_-|=\left|z + \frac{1-\sigma_H^-(1+\nu f_H^-)}{d_H^-}\right| = \left|z' + \pi^2 + \frac{1-\sigma_H^-(1+\nu f_H^-)}{d_H^-}\right| = \left|z' + \pi^2 - \frac{\sigma_H^-(1+\nu f_H^-)-1}{d_H^-}\right|.$$
From assumption \eqref{hyp const}, we have
$$\pi^2 - \frac{\sigma_H^-(1+\nu f_H^-)-1}{d_H^-} \geqslant \pi^2 - r_0 > 0,$$
and due to \refP{Prop DoreLabbas}, we obtain 
$$\begin{array}{lll}
\dis |z_-| & = & \dis \left|z' + \pi^2 - \frac{\sigma_H^-(1+\nu f_H^-)-1}{d_H^-}\right| \\ \ecart

& \geqslant & \dis \left(|z'| + \left|\pi^2 - \frac{\sigma_H^-(1+\nu f_H^-)-1}{d_H^-}\right|\right) \cos\left(\frac{\arg(z')}{2}\right) \\ \ecart

& \geqslant & \dis \left|\pi^2 - \frac{\sigma_H^-(1+\nu f_H^-)-1}{d_H^-}\right|\,\cos\left(\frac{\arg(z')}{2}\right) \\ \ecart

& \geqslant & \dis (\pi^2 - r_0) \frac{\sqrt{3}}{2}.
\end{array}$$
Now, if
$$\frac{1-\sigma_H^-(1+\nu f_H^-)}{d_H^-} > 0,$$
then
$$|z_-| = \left|z + \frac{1-\sigma_H^-(1+\nu f_H^-)}{d_H^-}\right| = \left|z' + \pi^2 + \frac{1-\sigma_H^-(1+\nu f_H^-)}{d_H^-}\right|,$$
and from \refP{Prop DoreLabbas}, it follows that
$$\begin{array}{lll}
\dis |z_-| & = & \dis \left|z' + \pi^2 + \frac{1-\sigma_H^-(1+\nu f_H^-)}{d_H^-}\right| \\ \ecart

& \geqslant & \dis \left(|z'| + \left|\pi^2 + \frac{1-\sigma_H^-(1+\nu f_H^-)}{d_H^-}\right|\right) \cos\left(\frac{\arg(z')}{2}\right) \\ \ecart

& \geqslant & \dis \left(\pi^2 + \frac{1-\sigma_H^-(1+\nu f_H^-)}{d_H^-}\right)\,\cos\left(\frac{\arg(z')}{2}\right) \\ \ecart

& \geqslant & \dis \pi^2 \,\cos\left(\frac{\arg(z')}{2}\right) \\ \ecart

& \geqslant & \dis (\pi^2 - r_0) \cos\left(\frac{\pi}{6}\right) \\ \ecart

& \geqslant & \dis (\pi^2 - r_0) \frac{\sqrt{3}}{2}.
\end{array}$$
Note that in this case, we have not used assumption \eqref{hyp const}.

Similarly, we deduce that
$$|z_+| = \left|z + \frac{1-\sigma_H^+(1 + f_H^+)}{d_H^+}\right| \geqslant (\pi^2 - r_0) \frac{\sqrt{3}}{2};$$
thus $z_\pm \in S_{(\pi-\varepsilon)/3}$. We then obtain
$$\begin{array}{lll}
\dis \left|z - \lambda_\pm\right| = \left|z_\pm + \frac{\lambda}{d_H^\pm}\right| & \geqslant & \dis \left(|z_\pm| + \left| \frac{\lambda}{d_H^\pm}\right|\right) \left|\cos\left(\frac{\arg(z_\pm) - \arg(\lambda)}{2}\right)\right| \\ \ecart

& \geqslant & \dis |z_\pm| \left|\cos\left(\frac{\arg(z_\pm) - \arg(\lambda)}{2}\right)\right| \\ \ecart

& \geqslant & \dis \frac{\sqrt{3}\,(\pi^2 - r_0)}{2} \left|\cos\left(\frac{\arg(z_\pm) - \arg(\lambda)}{2}\right)\right|.
\end{array}$$
On the other hand
$$\left|\arg(z_\pm) - \arg(\lambda)\right| \leqslant |\arg(z_\pm)| + |\arg(\lambda)| < \frac{\pi - \varepsilon}{3} + \frac{2(\pi - \varepsilon)}{3} = \pi - \varepsilon.$$
Hence
$$\begin{array}{lll}
\dis\left|z - \lambda_\pm\right| & \geqslant & \dis \frac{\sqrt{3}\,(\pi^2 - r_0)}{2} \left|\cos\left(\frac{\arg(z_\pm) - \arg(\lambda)}{2}\right)\right| \\ \ecart
& > & \dis \frac{\sqrt{3}\,(\pi^2 - r_0)}{2} \cos\left(\frac{\pi}{2} - \frac{\varepsilon}{2}\right) = \frac{\sqrt{3}\,(\pi^2 - r_0)}{2} \, \sin\left(\frac{\varepsilon}{2}\right) > 0.
\end{array}$$
\end{proof}

Let us go back to the function $f_{\lambda }$, given by \eqref{f_lambda (z)}, where $\lambda $ is fixed in $S_{2(\pi -\varepsilon)/3}$.

\begin{Prop}
Function $f_{\lambda }$, defined by \eqref{f_lambda (z)}, is analytic and bounded on $\pi ^{2}+S_{(\pi -\varepsilon)/3}$.
\end{Prop}

\begin{proof}
We first prove that the denominator in the expression of $f_\lambda$ never vanishes.

For $z \in \pi^2 + S_{(\pi - \varepsilon)/3}$, using \refL{Lem estim Zpm}, we have
$$\left|\sqrt{z-\lambda_-}\left(1 + e^{-2\ell\sqrt{z-\lambda_-}}\right) \left(1 - e^{-2L\sqrt{z-\lambda_+}}\right)\right| \geqslant \dfrac{\sqrt{3}\, (\pi^2-r_0)L}{2} \, \dfrac{\left(1 - e^{-\pi/2\sqrt{3}}\right)\sin(\varepsilon/2)}{1+ L \sqrt{|z - \lambda_+|}} >0.$$
We deduce that function $f_\lambda$ is holomorphic on $\pi^2 + S_{(\pi - \varepsilon)/3}$ and moreover it is bounded since
$$\left|f_\lambda(z)\right| \leqslant 1 + \frac{C \,\beta_S}{\beta_I \left(1 - e^{-\pi/2\sqrt{3}}\right)} \sqrt{\frac{z-\lambda_+}{z - \lambda_-}}\,\frac{1 + L \sqrt{|z - \lambda_+|}}{L \sqrt{|z - \lambda_+|}}, $$
where $C > 0$ is a constant independent of $z$ and $\lambda$.
\end{proof}

We conclude that 
\begin{equation*}
f_{\lambda }(-Q)=\Pi _{\lambda }.
\end{equation*}

\begin{Prop}
\label{Prop inv Det} Let $\varepsilon_0 \in \left(0,\dfrac{\pi}{8}\right)$. Then operator $\Pi _{\lambda}$ is boundedly invertible for all $\lambda \in S_{4(\pi -\varepsilon_0)/7}$. It follows that 
\begin{equation*}
D_{\lambda }^{-1}=\frac{1}{\beta_{I}}\left( I+e^{2\ell P_{\lambda
}^{-}}\right) ^{-1}\left( I-e^{2LP_{\lambda }^{+}}\right) ^{-1}\Pi _{\lambda
}^{-1}.
\end{equation*}
\end{Prop}
\begin{Rem}
To prove that operator $\L$ generates an analytic semigroup, it is necessary that 
$$\frac{4}{7}(\pi -\varepsilon_0) > \frac{\pi}{2},$$
which is true since $\varepsilon_0 < \dfrac{\pi}{8}$.  
\end{Rem}
\begin{proof}
The invertibility of $D_\lambda$ relies on the invertibility of $f_\lambda(-Q)$. We will now show that $f_\lambda(-Q)$ is invertible with bounded inverse by using \refP{Prop inv f}.
 
We recall the following notations for $\lambda \in S_{2(\pi-\varepsilon)/3}$ and $z \in \pi^2 + S_{(\pi-\varepsilon)/3}$
$$z_- = z + \frac{1-\sigma_H^-(1+\nu f_H^-)}{d_H^-} \quad \text{and} \quad z_+ = z + \frac{1-\sigma_H^+(1+f_H^+)}{d_H^+};$$
$$\lambda_- = -\frac{1-\sigma_H^-(1+\nu f_H^-)}{d_H^-} - \frac{\lambda}{d_H^-} \quad \text{and} \quad \lambda_+ = - \frac{1-\sigma_H^+(1+f_H^+)}{d_H^+} - \frac{\lambda}{d_H^+};$$
thus  $z_\pm \in S_{(\pi-\varepsilon)/3}$ and
$$z-\lambda_- = z_- + \frac{\lambda}{d_H^-} \quad \text{and} \quad z-\lambda_+ = z_+ + \frac{\lambda}{d_H^+}.$$
We consider two cases:
\begin{enumerate}
\item When $\lambda$ verifies: $-\dfrac{\pi}{3} + \dfrac{\varepsilon}{3} \leqslant \arg(\lambda) < \dfrac{2\pi}{3} - \dfrac{2\varepsilon}{3}$. 

From Proposition 13 in \cite{dlmmt2022}, we have
$$ -\frac{\pi}{3} + \dfrac{\varepsilon}{3} \leqslant \arg\left(z_\pm + \frac{\lambda}{d_H^\pm}\right) < \frac{2\pi}{3} - \dfrac{2\varepsilon}{3},$$
and
$$- \frac{\pi}{6} + \dfrac{\varepsilon}{6}\leqslant \arg\left(\sqrt{z - \lambda_\pm}\right) < \frac{\pi}{3} - \dfrac{\varepsilon}{3},$$
hence
$$-\frac{\pi}{2} + \dfrac{\varepsilon}{2}< \arg\left(\frac{\sqrt{z - \lambda_+}}{\sqrt{z - \lambda_-}} \right) = \arg\left(\sqrt{z - \lambda_+}\right) - \arg\left(\sqrt{z - \lambda_-}\right) < \frac{\pi}{2} - \dfrac{\varepsilon}{2}.$$
Since
$$\begin{array}{lll}
\dis \left|\arg\left(\frac{\left(1 - e^{-2\ell\sqrt{z-\lambda_-}}\right) \left(1 + e^{-2L\sqrt{z-\lambda_+}}\right)}{ \left(1 + e^{-2\ell\sqrt{z-\lambda_-}}\right) \left(1 - e^{-2L\sqrt{z-\lambda_+}}\right)}\right) \right| & \leqslant & \dis \left|\arg\left(\frac{\left(1 - e^{-2\ell\sqrt{z-\lambda_-}}\right)}{ \left(1 + e^{-2\ell\sqrt{z-\lambda_-}}\right)} \right)\right| \\ \ecart
&& \dis + \left|\arg\left( \frac{\left(1 - e^{-2L\sqrt{z-\lambda_+}}\right)}{\left(1 + e^{-2L\sqrt{z-\lambda_+}}\right)}\right)\right|,
\end{array}$$
from \refL{Lem estim Zpm} statement 2., we have
$$ \left|\arg\left(\frac{\left(1 - e^{-2\ell\sqrt{z-\lambda_-}}\right) \left(1 + e^{-2L\sqrt{z-\lambda_+}}\right)}{ \left(1 + e^{-2\ell\sqrt{z-\lambda_-}}\right) \left(1 - e^{-2L\sqrt{z-\lambda_+}}\right)}\right) \right| < \frac{2}{3}(\pi - \varepsilon).$$
For all $z \in \pi^2 + S_{(\pi-\varepsilon)/3}$, we note
$$\tilde{f}_\lambda(z) = \frac{\beta_S}{\beta_I} \frac{\sqrt{z-\lambda_+}}{\sqrt{z-\lambda_-}}\frac{\left(1 - e^{-2\ell\sqrt{z-\lambda_-}}\right) \left(1 + e^{-2L\sqrt{z-\lambda_+}}\right)}{ \left(1 + e^{-2\ell\sqrt{z-\lambda_-}}\right) \left(1 - e^{-2L\sqrt{z-\lambda_+}}\right)} = f_\lambda(z) - 1.$$
We deduce, from the previous inequalities, that
$$\left|\arg\left(\tilde{f}_\lambda (z)\right)\right| < \frac{7}{6}(\pi - \varepsilon),$$
and according to \refP{Prop DoreLabbas}, we have
$$\left|f_\lambda(z)\right| \geqslant \left(1 + \left|\tilde{f}_\lambda(z)\right|\right) \left|\cos\left(\frac{\arg\left(\tilde{f}_\lambda(z)\right)}{2}\right)\right| \geqslant \left|\cos\left(\frac{\arg\left(\tilde{f}_\lambda(z)\right)}{2}\right)\right|.$$
Now, assume that 
$$\varepsilon = \frac{\pi}{7} + \frac{6}{7}\, \varepsilon_0,$$
where
$$0 < \varepsilon_0 < \frac{\pi}{8}.$$
It follows that 
$$\left|f_\lambda(z)\right| \geqslant \left|\cos\left(\frac{\arg\left(\tilde{f}_\lambda(z)\right)}{2}\right)\right| > \cos\left(\frac{\pi}{2} - \frac{\varepsilon_0}{2}\right) = \sin\left(\frac{\varepsilon_0}{2}\right) > 0.$$
Therefore $1/f_\lambda$ is bounded and belongs to $H^\infty\left(\pi^2 + S_{(\pi-\varepsilon_0)/3}\right)$.

Finally, using again \refP{Prop inv f}, $f_\lambda(-Q) = \Pi_\lambda$ is boundedly invertible and 
$$\left[f_\lambda(-Q)\right]^{-1} = \frac{1}{f_\lambda}(-Q).$$

\item When $\lambda$ verifies: $-\dfrac{2\pi}{3} + \dfrac{2\varepsilon}{3} < \arg(\lambda) \leqslant \dfrac{\pi}{3} - \dfrac{\varepsilon}{3}$. 

By using an analogous method, we obtain the invertibility of $\Pi_\lambda$.
\end{enumerate}
To conclude, there exists a constant $C>0$ independent of $\lambda$ such that
\begin{equation}\label{majoration norme D_lambda -1}
\left\|D_\lambda^{-1}\right\|_{\L(X)} \leqslant C.
\end{equation}
\end{proof}

We are now in position to solve system \eqref{syst final}. Thanks to \refP{Prop inv Det}, $D_{\lambda }$ is invertible for $\lambda \in S_{4(\pi -\varepsilon _{0})/7}$ and the expressions of constants $\delta _{I}$ and $\gamma _{S}$ are given by
\begin{equation*}
\left\{ 
\begin{array}{lll}
\delta _{I} & = & \dis\beta _{S}\,D_{\lambda }^{-1}\,P_{\lambda
}^{+}(P_{\lambda }^{-})^{-1} \left( I+e^{2LP_{\lambda }^{+}}\right) \left(
R_{S}-R_{I}\right) \\ \ecart
&& \dis + D_{\lambda }^{-1}\,\left( I-e^{2LP_{\lambda}^{+}}\right) \left( \beta _{I}\,R_{I} + \beta _{S}\,P_{\lambda}^{+}(P_{\lambda }^{-})^{-1}R_{S}\right)  \\ 
\\
\gamma _{S} & = & \dis-\beta _{I}\,D_{\lambda }^{-1}\left( I+e^{2\ell
P_{\lambda }^{-}}\right) \left( R_{S}-R_{I}\right) \\ \ecart
&&\dis + D_{\lambda}^{-1}\,\left( I-e^{2\ell P_{\lambda }^{-}}\right) \left( \beta_{I}\,R_{I}+\beta _{S}\,P_{\lambda }^{+}(P_{\lambda }^{-})^{-1}R_{S}\right).
\end{array}
\right. 
\end{equation*}
We now explicit the solution $(h_{I},h_{S})$ by using the fact that operators $D_{\lambda }^{-1}$, $P_{\lambda }^{\pm }$ and $(P_{\lambda }^{\pm})^{-1}$ commute among themselves.

For $x\in (-\ell ,0)$, we have 
\begin{equation*}
\begin{array}{llll}
h_{I}(x) & = & 2\,\beta _{S}\,D_{\lambda }^{-1}\,P_{\lambda}^{+}(P_{\lambda}^{-})^{-1}e^{-xP_{\lambda }^{-}}\left( I-e^{2(x+\ell)P_{\lambda }^{-}}\right) R_{S} &  \\ \ecart 
&  & -\beta _{S}\,D_{\lambda }^{-1}\,P_{\lambda }^{+}(P_{\lambda}^{-})^{-1}e^{-xP_{\lambda }^{-}}\left( I-e^{2(x+\ell )P_{\lambda}^{-}}\right) \left( I+e^{2LP_{\lambda }^{+}}\right) R_{I} &  \\ \ecart 
&  & +\beta_{I}\,D_{\lambda }^{-1}e^{-xP_{\lambda }^{-}}\left(I-e^{2(x+\ell )P_{\lambda }^{-}}\right) \left( I-e^{2LP_{\lambda}^{+}}\right) R_{I} &  \\ \ecart 
&  & -e^{(x+\ell )P_{\lambda }^{-}}w_{I}(g_{I})(-\ell ) &  \\ \ecart 
&  & +w_{I}(g_{I})(x),  
\end{array}
\end{equation*}%
where 
\begin{equation*}
w_{I}(g_{I})(x)=\frac{1}{2}\int_{-\ell }^{x}e^{(x-t)P_{\lambda
}^{-}}(P_{\lambda }^{-})^{-1}g_{I}(t)~dt+\frac{1}{2}\int_{x}^{0}e^{(t-x)P_{%
\lambda }^{-}}(P_{\lambda }^{-})^{-1}g_{I}(t)~dt;
\end{equation*}%
and for $x\in (0,L)$, we have 
\begin{equation*}
\begin{array}{llll}
h_{S}(x) & = & 2\,\beta _{I}\,D_{\lambda }^{-1}e^{xP_{\lambda }^{+}}\left(
I-e^{2(L-x)P_{\lambda }^{+}}\right) R_{I} &  \\ \ecart 
&  & -\beta _{I}\,D_{\lambda }^{-1}e^{xP_{\lambda }^{+}}\left(I-e^{2(L-x)P_{\lambda }^{+}}\right) \left( I+e^{2\ell P_{\lambda}^{-}}\right) R_{S} &  \\ \ecart 
&  & +\beta _{S}\,D_{\lambda }^{-1}\,P_{\lambda }^{+}(P_{\lambda
}^{-})^{-1}e^{xP_{\lambda }^{+}}\left( I-e^{2(L-x)P_{\lambda }^{+}}\right)
\left( I-e^{2\ell P_{\lambda }^{-}}\right) R_{S} &  \\ 
\ecart &  & -e^{(L-x)P_{\lambda }^{+}}w_{S}(g_{S})(L) &  \\ 
\ecart &  & +w_{S}(g_{S})(x), & 
\end{array}%
\end{equation*}%
where 
\begin{equation*}
w_{S}(g_{S})(x)=\frac{1}{2}\int_{0}^{x}e^{(x-t)P_{\lambda }^{+}}(P_{\lambda
}^{+})^{-1}g_{S}(t)~dt+\frac{1}{2}\int_{x}^{L}e^{(t-x)P_{\lambda
}^{+}}(P_{\lambda }^{+})^{-1}g_{S}(t)~dt,
\end{equation*}% 
\begin{equation*}
R_{I}=w_{I}(g_{I})(0)-e^{\ell P_{\lambda }^{-}}w_{I}(g_{I})(-\ell )=\frac{1}{%
2}\int_{-\ell }^{0}e^{-tP_{\lambda }^{-}}\left( I-e^{2(t+\ell )P_{\lambda
}^{-}}\right) (P_{\lambda }^{-})^{-1}g_{I}(t)~dt,
\end{equation*}%
and 
\begin{equation*}
R_{S}=w_{S}(g_{S})(0)-e^{LP_{\lambda }^{+}}w_{S}(g_{S})(L)=\frac{1}{2}%
\int_{0}^{L}e^{tP_{\lambda }^{+}}\left( I-e^{2(L-t)P_{\lambda }^{+}}\right)
(P_{\lambda }^{+})^{-1}g_{S}(t)~dt.
\end{equation*}

\begin{Rem}
Note that in the representation formula of the solution $h_{I}$, the first three terms express, in particular, the effect of the transmission conditions
between the two habitats; the fourth term expresses the effect of a boundary
condition in $-\ell $ and the fifth term, obviously, expresses the effect of
the direct and retrograde evolution inside the domain $\Omega _{I}$.

The same comment is valid for the representation formula of $h_{S}$.
\end{Rem}

\subsection{Maximal regularity of the solution}

From \refP{Prop Qpm}, for all $\eta >0$, operators $-Q$ and $-Q^{\pm }$ are sectorial of angle $\eta $. On the other hand, by Proposition 3.1, p. 191 in \cite{labbas-moussaoui2000}, we have $-Q\in $ BIP$\,(X,\eta )$. Moreover, due to \eqref{hyp const} and Theorem~2.3, p. 69 in \cite{arendt-bu-haase2001}, we deduce that $-Q^{\pm }\in $ BIP$\,(X,\eta )$. Thus, thanks to Theorem 2.4, p. 408 in \cite{monniaux1997}, we obtain that 
\begin{equation*}
-Q^{\pm }+\frac{\lambda }{d_{H}^{\pm }}\in \text{BIP}(X,\theta ),
\end{equation*}%
where 
\begin{equation*}
\theta =\max \left( \eta ,|\arg (\lambda )|\right) =|\arg (\lambda )|<\frac{4(\pi -\varepsilon _{0})}{7}.
\end{equation*}%
Finally, due to Proposition 3.2.1, e), p. 71 in \cite{haase2006}, we deduce that 
\begin{equation*}
-P_{\lambda }^{\pm }\in \text{BIP}(X,2(\pi -\varepsilon _{0})/7).
\end{equation*}
Recall that for all fixed $c>0$, we have 
\begin{equation}
\forall \,\psi \in X,\quad e^{cP_{\lambda }^{\pm }}\psi \in D((P_{\lambda
}^{\pm })^{\infty })=\bigcap_{k\geqslant 0}D((P_{\lambda }^{\pm })^{k})=D(Q^{\infty
}).
\label{exp(cG) et exp(cH)}
\end{equation}%
For the maximal regularity of $h_{I}$, we have to show that 
\begin{equation*}
h_{I}\in W^{2,p}(-\ell ,0;X)\cap L^{p}\left(-\ell ,0;D\left((P_{\lambda}^{-})^{2}\right)\right).
\end{equation*}
It suffices to show, for instance, that 
\begin{equation*}
h_{I}\in L^{p}\left(-\ell ,0;D\left((P_{\lambda }^{-})^{2}\right)\right).
\end{equation*}%
We recall that, for almost every $x\in (-\ell ,0)$, we have 
\begin{equation*}
\begin{array}{llll}
h_{I}(x) & = & 2\,\beta _{S}\,D_{\lambda }^{-1}\,P_{\lambda}^{+}(P_{\lambda }^{-})^{-1}e^{-xP_{\lambda }^{-}}\left( I-e^{2(x+\ell)P_{\lambda }^{-}}\right) R_{S} &  \\ \ecart 
&  & -\beta _{S}\,D_{\lambda }^{-1}\,P_{\lambda }^{+}(P_{\lambda}^{-})^{-1}e^{-xP_{\lambda }^{-}}\left( I-e^{2(x+\ell )P_{\lambda}^{-}}\right) \left( I+e^{2LP_{\lambda }^{+}}\right) R_{I} &  \\ \ecart 
&  & +\beta _{I}\,D_{\lambda }^{-1}e^{-xP_{\lambda }^{-}}\left(I-e^{2(x+\ell )P_{\lambda }^{-}}\right) \left( I-e^{2LP_{\lambda}^{+}}\right) R_{I} & \\ \ecart 
&  & -e^{(x+\ell )P_{\lambda }^{-}}w_{I}(g_{I})(-\ell )  \\ \ecart 
&  & +w_{I}(g_{I})(x). 
\end{array}
\end{equation*}
The last term is directly treated using \refT{Th DorVen}, thus 
\begin{equation*}
w_{I}(g_{I})\in L^{p}\left(-\ell ,0;D\left((P_{\lambda }^{-})^{2}\right)\right).
\end{equation*}%
Moreover, we have 
\begin{equation*}
\begin{array}{lll}
\dis e^{(x+\ell )P_{\lambda }^{-}}w_{I}(g_{I})(-\ell ) & = & \dis\frac{1}{2}%
e^{(x+\ell )P_{\lambda }^{-}}\int_{-\ell }^{0}e^{(s+\ell )P_{\lambda
}^{-}}(P_{\lambda }^{-})^{-1}g(s)~ds \\ \\
& = & \dis\frac{1}{2}(P_{\lambda }^{-})^{-1}\int_{-\ell
}^{x}e^{(x-s)P_{\lambda }^{-}}e^{2(s+\ell )P_{\lambda }^{-}}g(s)~ds \\ 
\ecart 
&& \dis+\frac{1}{2}(P_{\lambda }^{-})^{-1}e^{2(x+\ell )P_{\lambda
}^{-}}\int_{x}^{0}e^{(s-x)P_{\lambda }^{-}}g(s)~ds.%
\end{array}%
\end{equation*}%
Since 
\begin{equation*}
s\longmapsto e^{2(s+\ell )P_{\lambda }^{-}}g(s)\in L^{p}(-\ell ,0;X),
\end{equation*}%
using again \refT{Th DorVen},  we obtain that 
\begin{equation*}
x\longmapsto e^{(x+\ell )P_{\lambda }^{-}}w_{I}(g_{I})(-\ell )\in
L^{p}\left(-\ell ,0;D\left((P_{\lambda }^{-})^{2}\right)\right).
\end{equation*}
Let us, for instance, analyze the following first term, since all the other terms can be treated similarly:
\begin{equation*}
\begin{array}{l}
2\,\beta_{S}\,D_{\lambda }^{-1}\,P_{\lambda }^{+}(P_{\lambda
}^{-})^{-1}e^{-xP_{\lambda }^{-}}\left( I-e^{2(x+\ell )P_{\lambda
}^{-}}\right) R_{S} \\ \ecart

=2\,\beta _{S}\,D_{\lambda }^{-1}\,P_{\lambda }^{+}(P_{\lambda
}^{-})^{-1}e^{-xP_{\lambda }^{-}}\left( I-e^{2(x+\ell )P_{\lambda
}^{-}}\right) w_{S}(g_{S})(0) \\ \ecart

\quad -2\,\beta _{S}\,D_{\lambda }^{-1}\,P_{\lambda }^{+}(P_{\lambda
}^{-})^{-1}e^{-xP_{\lambda }^{-}}\left( I-e^{2(x+\ell )P_{\lambda
}^{-}}\right) e^{LP_{\lambda }^{+}}w_{S}(g_{S})(L).%
\end{array}%
\end{equation*}%
From \eqref{exp(cG) et exp(cH)}, we just have to study 
\begin{equation*}
\begin{array}{l}
\left( P_{\lambda }^{-}\right) ^{2}D_{\lambda }^{-1}\,P_{\lambda
}^{+}(P_{\lambda }^{-})^{-1}e^{-xP_{\lambda }^{-}}\left( I-e^{2(x+\ell
)P_{\lambda }^{-}}\right) w_{S}(g_{S})(0) \\ 
\ecart=\dis\frac{1}{2}\,D_{\lambda }^{-1}\left( I-e^{2(x+\ell )P_{\lambda
}^{-}}\right) P_{\lambda }^{-}e^{-xP_{\lambda
}^{-}}\int_{0}^{L}e^{tP_{\lambda }^{+}}g_{S}(t)~dt.%
\end{array}%
\end{equation*}%
According to Lemma 18, p. 19 in \cite{dlmmt2022}, we have 
\begin{equation}
\left\{ 
\begin{array}{ll}
\dis\int_{0}^{L}e^{tP_{\lambda }^{+}}g_{S}(t)~dt & \in \left( D(P_{\lambda
}^{+}),X\right) _{\frac{1}{p},p} \\ 
\ecart\dis\int_{0}^{L}e^{(L-t)P_{\lambda }^{+}}g_{S}(t)~dt & \in \left(
D(P_{\lambda }^{+}),X\right) _{\frac{1}{p},p} \\ 
\ecart\dis\int_{-\ell }^{0}e^{-tP_{\lambda }^{-}}g_{I}(t)~dt & \in \left(
D(P_{\lambda }^{-}),X\right) _{\frac{1}{p},p} \\ 
\ecart\dis\int_{-\ell }^{0}e^{(t+\ell )P_{\lambda }^{-}}g_{I}(t)~dt & \in
\left( D(P_{\lambda }^{-}),X\right) _{\frac{1}{p},p}.%
\end{array}%
\right. 
\end{equation}%
Since $D(P_{\lambda }^{+})=D(P_{\lambda }^{-})=D(\sqrt{-Q})$, it follows that 
\begin{equation*}
\left( D(P_{\lambda }^{+}),X\right) _{\frac{1}{p},p}=\left( D(P_{\lambda
}^{-}),X\right) _{\frac{1}{p},p}=\left( D\left(\sqrt{-Q}\right),X\right) _{\frac{1}{p}%
,p}=\left( D(Q),X\right) _{\frac{1}{2}+\frac{1}{2p},p}.
\end{equation*}
Thus, due to \eqref{Reg triebel}, for $n=1$, we deduce that 
\begin{equation*}
x\longmapsto P_{\lambda }^{-}e^{-xP_{\lambda
}^{-}}\int_{0}^{L}e^{tP_{\lambda }^{+}}g_{S}(t)~dt\in L^{p}(-\ell ,0;X);
\end{equation*}
hence
\begin{equation*}
h_{I}\in L^{p}\left(-\ell ,0;D\left((P_{\lambda }^{-})^{2}\right)\right).
\end{equation*}
In the same way, we obtain that 
\begin{equation*}
h_{S}\in W^{2,p}(0,L;X)\cap L^{p}\left(0,L;D\left((P_{\lambda }^{+})^{2}\right)\right).
\end{equation*}

\section{Estimation of the norm of the resolvent operator}

\begin{Lem}
\label{Lem LMT} Let $T$ be a linear operator such that $-T\in $ BIP$(X,\theta
_{T})$, with $\theta _{T}\in \lbrack 0,\pi /2[$ and $0\in \rho (T)$. Let $%
g\in L^{p}(a,b;X)$ with $1<p<+\infty $ and $a<b$. For all $\theta \in
\;]0,\pi -\theta _{T}[$, $\mu \in \overline{S_{\theta }}\subset \rho (-T)$ and $x\in \lbrack a,b]$, we set
\begin{equation*}
I_{\mu ,g}(x)=\int_{a}^{x}e^{-(x-s)\sqrt{-T+\mu I}}g(s)~ds\quad \text{and}%
\quad J_{\mu ,g}(x)=\int_{x}^{b}e^{-(s-x)\sqrt{-T+\mu I}}g(s)~ds.
\end{equation*}%
Then, we have 
\begin{equation*}
\left\Vert I_{\mu ,f}\right\Vert _{L^{p}(a,b;X)}\leqslant \frac{C}{\sqrt{%
1+|\mu |}}\,\left\Vert f\right\Vert _{L^{p}(a,b;X)}\quad \text{and}\quad
\left\Vert J_{\mu ,f}\right\Vert _{L^{p}(a,b;X)}\leqslant \frac{C}{\sqrt{%
1+|\mu |}}\,\left\Vert f\right\Vert _{L^{p}(a,b;X)},
\end{equation*}%
where $C>0$ is independent of $g$ and $\mu $.
\end{Lem}

This result is a consequence of Lemma 4.6 in \cite{flmt2021} and Lemma 4.11 in \cite{lmt2022}.

\begin{Lem}
\label{Lem LMT 2} Let $g\in L^{p}(a,b;X)$, $1<p<+\infty $ and $T$ be a linear operator such that $-T\in $ BIP$(X,\theta_{T})$, where $\theta _{T}\in \lbrack 0,\pi /2[$ and $0\in \rho (T)$. Let $\theta \in \;]0,\pi -\theta _{T}[$ fixed. Then, there exists $C>0$, such that, for all $\eta ,\mu \in \overline{S_{\theta }}\subset \rho (-T)$,  we have

\begin{enumerate}
\item $\dis \left\|e^{-(.-a)\sqrt{-T + \eta I}} \int_a^b e^{-(s-a)\sqrt{%
-T+\mu I}} g(s)~ ds\right\|_{L^p(a,b;X)} \leqslant \left(\frac{C}{\sqrt{%
1+|\mu|}} + \frac{C}{\sqrt{1+|\eta|}} \right)\left\|g\right\|_{L^p(a,b;X)},$

\item $\dis \left\|e^{-(.-a)\sqrt{-T + \eta I}} \int_a^b e^{-(b-s)\sqrt{%
-T+\mu I}} g(s)~ ds \right\|_{L^p(a,b;X)} \leqslant \left(\frac{C}{\sqrt{%
1+|\mu|}} + \frac{C}{\sqrt{1+|\eta|}} \right) \left\|g\right\|_{L^p(a,b;X)},$

\item $\dis \left\|e^{-(b-.)\sqrt{-T + \eta I}} \int_a^b e^{-(b-s)\sqrt{%
-T+\mu I}} g(s)~ ds\right\|_{L^p(a,b;X)} \leqslant \left(\frac{C}{\sqrt{%
1+|\mu|}} + \frac{C}{\sqrt{1+|\eta|}} \right) \left\|g\right\|_{L^p(a,b;X)},$

\item $\dis \left\|e^{-(b-.)\sqrt{-T + \eta I}} \int_a^b e^{-(s-a)\sqrt{%
-T+\mu I}} g(s)~ ds\right\|_{L^p(a,b;X)} \leqslant \left(\frac{C}{\sqrt{%
1+|\mu|}} + \frac{C}{\sqrt{1+|\eta|}} \right)\left\|g\right\|_{L^p(a,b;X)}.$
\end{enumerate}
\end{Lem}

This Lemma is proved in \cite{lmt2022}, Lemma 4.12, p. 33.

\begin{Rem}
In our case, $T=Q^{+}$ or $T=Q^{-}$; likewise $\eta $ and $\mu $ will be
replaced by $\lambda /d_{H}^{-}$ and $\lambda /d_{H}^{+}$.
\end{Rem}

According to \eqref{estimation Q_i}, there exists $C>0$, such that for all $\lambda \in S_{12(\pi -\varepsilon _{0})/21}$
\begin{equation}
\left\Vert (P_{\lambda }^{-})^{-1}\right\Vert _{\L (X)}\leqslant \frac{C}{%
\sqrt{1+|\lambda |}}\quad \text{and}\quad \left\Vert (P_{\lambda
}^{+})^{-1}\right\Vert _{\L (X)}\leqslant \frac{C}{\sqrt{1+|\lambda |}}.
\label{estim G-1 et H-1}
\end{equation}%
Moreover, for $\alpha \in \RR$ and $t_{0}>0$ fixed, due to \cite{dore_yakubov2000}, Lemma~2.6, (b), p. 104, there exist $K,C,\omega >0$, such that 
\begin{equation}
\left\Vert (-P_{\lambda }^{\pm })^{\alpha }\,e^{t_{0}P_{\lambda }^{\pm
}}\right\Vert _{\L (X)}\leqslant Ke^{-t_{0}\omega \sqrt{|\lambda
|/d_{H}^{\pm }}}\leqslant Ke^{-C\sqrt{|\lambda |}}.  \label{estim DY}
\end{equation}

\begin{Lem}\label{LemGH-1etHG-1} 
There exists $C>0$ such that for all $\lambda \in
S_{4(\pi -\varepsilon _{0})/7}$ 
\begin{equation*}
\Vert P_{\lambda }^{-}(P_{\lambda }^{+})^{-1}\Vert _{\L (X)}\leqslant C\quad 
\text{and}\quad \Vert P_{\lambda }^{+}(P_{\lambda }^{-})^{-1}\Vert _{\L %
(X)}\leqslant C.
\end{equation*}
\end{Lem}

\begin{proof}
We have 
$$\begin{array}{lll}
P_\lambda^- (P_\lambda^+)^{-1} &=& \dis (P_\lambda^-)^2 (P_\lambda^-)^{-1} (P_\lambda^+)^{-1} = \left(-Q^- + \frac{\lambda}{d_H^-} I\right)(P_\lambda^-)^{-1}(P_\lambda^+)^{-1} \\ \ecart
& =& \dis -Q (P_\lambda^-)^{-1} (P_\lambda^+)^{-1} + \left(\frac{1-\sigma_H^-(1+\nu f_H^-)}{d_H^-} + \frac{\lambda}{d_H^-}\right) (P_\lambda^-)^{-1} (P_\lambda^+)^{-1} \\ \ecart
&=& \dis \sqrt{-Q} (P_\lambda^-)^{-1} \sqrt{-Q} (P_\lambda^+)^{-1} + \left(\frac{1-\sigma_H^-(1+\nu f_H^-)}{d_H^-} + \frac{\lambda}{d_H^-}\right) (P_\lambda^-)^{-1} (P_\lambda^+)^{-1},
\end{array}$$
hence
$$\begin{array}{lll}
\dis \|P_\lambda^- (P_\lambda^+)^{-1}\|_{\L(X)} & \leqslant & \dis \| \sqrt{-Q}(P_\lambda^-)^{-1}\|_{\L(X)} \|\sqrt{-Q} (P_\lambda^+)^{-1}\|_{\L(X)} \\ \ecart
&& \dis + \left(\frac{|1-\sigma_H^-(1+\nu f_H^-)|}{d_H^-} + \frac{|\lambda|}{d_H^-}\right) \|(P_\lambda^-)^{-1}\|_{\L(X)} \|(P_\lambda^+)^{-1}\|_{\L(X)}.
\end{array}$$
From \eqref{estim G-1 et H-1}, there exists $C>0$, independent of $\lambda$, such that
$$\frac{|1-\sigma_H^-(1+\nu f_H^-)|}{d_H^-} \, \|(P_\lambda^-)^{-1}\|_{\L(X)} \|(P_\lambda^+)^{-1}\|_{\L(X)} \leqslant \frac{C}{1 + |\lambda|} \leqslant C,$$
and
$$\frac{|\lambda|}{d_H^-} \, \|(P_\lambda^-)^{-1}\|_{\L(X)} \|(P_\lambda^+)^{-1}\|_{\L(X)} \leqslant \frac{C \,|\lambda|}{1+ |\lambda|} \leqslant C.$$
We have
$$\sqrt{-Q} (P_\lambda^-)^{-1} = -(-Q)^{1/2} \left( -Q + \frac{1-\sigma _{H}^{-}(1+\nu f_H^-)}{d_{H}^{-}}I + \frac{\lambda }{d_{H}^{-}}I \right) ^{-1/2},$$
and
$$\sqrt{-Q} (P_\lambda^+)^{-1} = -(-Q)^{1/2} \left( -Q + \frac{1-\sigma_{H}^{+}(1+f_H^+)}{d_{H}^{+}}I + \frac{\lambda }{d_{H}^{+}}I \right)^{-1/2},$$
then using \cite{dore_yakubov2000}, Lemma~2.6, (a), p. 104, there exists $C > 0$ such that
$$\|\sqrt{-Q} (P_\lambda^-)^{-1}\|_{\L(X)} \leqslant C \quad \text{and} \quad
\|\sqrt{-Q}(P_\lambda^+)^{-1}\|_{\L(X)} \leqslant C.$$
In the same way, we obtain that
$$\|P_\lambda^+ (P_\lambda^-)^{-1}\|_{\L(X)} \leqslant C.$$
\end{proof}
We recall that the solution $h_{I}$ is given by: 
\begin{equation*}
\begin{array}{lll}
h_{I}(x) & = & 2\,\beta _{S}\,D_{\lambda }^{-1}\,P_{\lambda}^{+}(P_{\lambda }^{-})^{-1}e^{-xP_{\lambda }^{-}}\left( I-e^{2(x+\ell)P_{\lambda }^{-}}\right) R_{S}  \\ \ecart 
&& -\beta _{S}\,D_{\lambda }^{-1}\,P_{\lambda }^{+}(P_{\lambda}^{-})^{-1} e^{-x P_{\lambda }^{-}}\left( I-e^{2(x+\ell )P_{\lambda}^{-}}\right) \left(I+e^{2LP_{\lambda }^{+}}\right) R_{I}  \\ \ecart 
&& +\beta _{I}\,D_{\lambda }^{-1}e^{-xP_{\lambda }^{-}}\left(I-e^{2(x+\ell) P_{\lambda}^{-}} \right) \left( I-e^{2LP_{\lambda}^{+}}\right) R_{I}  \\ \ecart 
&& -e^{(x+\ell )P_{\lambda }^{-}}w_{I}(g_{I})(-\ell ) \\ \ecart 
&& +w_{I}(g_{I})(x),
\end{array}
\end{equation*}
where 
\begin{equation*}
w_{I}(g_{I})(x)=\frac{1}{2}(P_{\lambda }^{-})^{-1}\int_{-\ell
}^{x}e^{(x-t)P_{\lambda }^{-}}g_{I}(t)~dt+\frac{1}{2}(P_{\lambda
}^{-})^{-1}\int_{x}^{0}e^{(t-x)P_{\lambda }^{-}}g_{I}(t)~dt,
\end{equation*}
and 
\begin{equation*}
\left\{ 
\begin{array}{lll}
R_{I} & = & \dis\frac{1}{2}(P_{\lambda }^{-})^{-1}\int_{-\ell
}^{0}e^{-tP_{\lambda }^{-}}\left( I-e^{2(t+\ell )P_{\lambda }^{-}}\right)
g_{I}(t)~dt \\ 
\ecart R_{S} & = & \dis\frac{1}{2}(P_{\lambda
}^{+})^{-1}\int_{0}^{L}e^{tP_{\lambda }^{+}}\left( I-e^{2(L-t)P_{\lambda
}^{+}}\right) g_{S}(t)~dt.%
\end{array}%
\right. 
\end{equation*}
Since $P_{\lambda }^{+}$ and $P_{\lambda }^{-}$ generate bounded
analytic semigroups, there exists $C>0$, independent of $%
\lambda $, such that 
\begin{equation*}
\left\Vert I-e^{2(x+\ell )P_{\lambda }^{-}}\right\Vert _{\L (X)}\leqslant
C,\quad \left\Vert I-e^{2LP_{\lambda }^{+}}\right\Vert _{\L (X)}\leqslant
C\quad \text{and}\quad \left\Vert I+e^{2LP_{\lambda }^{+}}\right\Vert _{\L %
(X)}\leqslant C.
\end{equation*}
Then, for almost every $x\in (-\ell ,0)$, we have 
\begin{equation*}
\begin{array}{lll}
\Vert h_{I}(x)\Vert _{X} & \leqslant  & 2\,\beta _{S}\Vert D_{\lambda}^{-1}\Vert _{\L (X)}\Vert P_{\lambda }^{+}(P_{\lambda }^{-})^{-1}\Vert _{\L(X)}\left\Vert I-e^{2(x+\ell )P_{\lambda }^{-}}\right\Vert _{\L (X)}\Vert e^{-xP_{\lambda}^{-}}R_{S}\Vert _{X}  \\ \ecart 
&& +\beta _{S}\,\Vert D_{\lambda }^{-1}\Vert _{\L (X)}\Vert P_{\lambda }^{+}(P_{\lambda }^{-})^{-1}\Vert _{\L (X)}\left\Vert I-e^{2(x+\ell )P_{\lambda }^{-}}\right\Vert _{\L (X)}\left\Vert I+e^{2LP_{\lambda }^{+}}\right\Vert _{\L (X)}\Vert e^{-xP_{\lambda}^{-}}R_{I}\Vert_{X} \\  \ecart 
&& +\beta _{I}\Vert D_{\lambda }^{-1}\Vert _{\L (X)}\left\Vert I-e^{2(x+\ell )P_{\lambda }^{-}}\right\Vert _{\L (X)}\left\Vert I-e^{2LP_{\lambda }^{+}}\right\Vert _{\L (X)}\Vert e^{-xP_{\lambda}^{-}}R_{I}\Vert _{X}   \\ \ecart &  & +\Vert e^{(x+\ell )P_{\lambda }^{-}}w_{I}(g_{I})(-\ell )\Vert_{X}+\Vert w_{I}(g_{I})(x)\Vert _{X}  \\ \\

& \leqslant  & C\left( \Vert e^{-xP_{\lambda }^{-}}R_{S}\Vert _{X}+\Vert e^{-xP_{\lambda }^{-}}R_{I}\Vert _{X}\right) +\Vert e^{(x+\ell )P_{\lambda}^{-}}w_{I}(g_{I})(-\ell )\Vert _{X}+\Vert w_{I}(g_{I})(x)\Vert _{X}; 
\end{array}
\end{equation*}
hence
\begin{equation*}
\begin{array}{lll}
\Vert h_{I}\Vert _{L^{p}(-\ell ,0;X)} & \leqslant  & \dis C\left( \Vert
e^{-\cdot P_{\lambda }^{-}}R_{I}\Vert _{L^{p}(-\ell ,0;X)}+\Vert e^{-\cdot
P_{\lambda }^{-}}R_{S}\Vert _{L^{p}(-\ell ,0;X)}\right)  \\ 
\ecart &  & \dis+\Vert e^{(\cdot +\ell )P_{\lambda }^{-}}w_{I}(g_{I})(-\ell
)\Vert _{L^{p}(-\ell ,0;X)}+\Vert w_{I}(g_{I})(\cdot )\Vert _{L^{p}(-\ell
,0;X)}. 
\end{array}
\end{equation*}
From \eqref{estim G-1 et H-1} and \refL{Lem LMT 2}, there exists $C>0$, independent of $\lambda $, such that 
\begin{equation*}
\begin{array}{lll}
\Vert e^{-\cdot P_{\lambda }^{-}}R_{I}\Vert _{L^{p}(-\ell ,0;X)} & \leqslant 
& \dis\frac{C}{\sqrt{1+|\lambda |}}\left\Vert e^{-\cdot P_{\lambda
}^{-}}\int_{-\ell }^{0}e^{-tP_{\lambda }^{-}}\left( I-e^{2(t+\ell
)P_{\lambda }^{-}}\right) g_{I}(t)~dt\right\Vert _{L^{p}(-\ell ,0;X)} \\ 
\ecart & \leqslant  & \dis\frac{C}{1+|\lambda |}\left\Vert \left(
I-e^{2(\cdot +\ell )P_{\lambda }^{-}}\right) g_{I}\right\Vert _{L^{p}(-\ell
,0;X)} \\ 
\ecart & \leqslant  & \dis\frac{C}{1+|\lambda |}\Vert g_{I}\Vert
_{L^{p}(-\ell ,0;X)}.
\end{array}
\end{equation*}
Similarly, we have 
\begin{equation*}
\begin{array}{lll}
\Vert e^{-\cdot P_{\lambda }^{-}}R_{S}\Vert _{L^{p}(-\ell ,0;X)} & \leqslant 
& \dis\frac{C}{\sqrt{1+|\lambda |}}\left\Vert e^{-\cdot P_{\lambda
}^{-}}\int_{0}^{L}e^{tP_{\lambda }^{+}}\left( I-e^{2(L-t)P_{\lambda
}^{+}}\right) g_{S}(t)~dt\right\Vert _{L^{p}(0,L;X)} \\ 
\ecart & \leqslant  & \dis\frac{C}{1+|\lambda |}\left\Vert \left(
I-e^{2(L-\cdot )P_{\lambda }^{+}}\right) g_{S}\right\Vert _{L^{p}(0,L;X)} \\ 
\ecart & \leqslant  & \dis\frac{C}{1+|\lambda |}\Vert g_{S}\Vert
_{L^{p}(0,L;X)},%
\end{array}%
\end{equation*}%
and 
\begin{equation*}
\Vert e^{(\cdot +\ell )P_{\lambda }^{-}}w_{I}(g_{I})(-\ell )\Vert
_{L^{p}(-\ell ,0;X)}\leqslant \frac{C}{1+|\lambda |}\Vert g_{I}\Vert
_{L^{p}(-\ell ,0;X)}.
\end{equation*}%
Then, due to \eqref{estim G-1 et H-1} and \refL{Lem LMT}, there exists $C>0$, independent of $\lambda $, such that 
\begin{equation*}
\Vert w_{I}(g_{I})(\cdot )\Vert _{L^{p}(-\ell ,0;X)}\leqslant \frac{C}{1+|\lambda |/d_{H}^{-}}\Vert g_{I}\Vert _{L^{p}(-\ell ,0;X)}\leqslant 
\frac{C}{1+|\lambda |}\Vert g_{I}\Vert _{L^{p}(-\ell ,0;X)}.
\end{equation*}
Finally, there exists $C>0$ such that, for all $\lambda \in S_{4(\pi -\varepsilon_{0})/7}$, we obtain 
\begin{equation*}
\Vert h_{I}\Vert _{L^{p}(-\ell ,0;X)}\leqslant \frac{C}{1+|\lambda |}\left(
\Vert g_{I}\Vert _{L^{p}(-\ell ,0;X)}+\Vert g_{S}\Vert
_{L^{p}(0,L;X)}\right) ,
\end{equation*}%
and 
\begin{equation*}
\Vert h_{S}\Vert _{L^{p}(0,L;X)}\leqslant \frac{C}{1+|\lambda |}\left( \Vert
g_{I}\Vert _{L^{p}(-\ell ,0;X)}+\Vert g_{S}\Vert _{L^{p}(0,L;X)}\right) ;
\end{equation*}%
hence
\begin{equation*}
\begin{array}{lll}
\dis\Vert h\Vert _{L^{p}(\Omega )} & = & \dis\Vert h_{I}\Vert _{L^{p}(\Omega
_{I})}+\Vert h_{S}\Vert _{L^{p}(\Omega _{S})} \\ \ecart 
& = & \dis\Vert h_{I}\Vert _{L^{p}(-\ell ,0;X)}+\Vert h_{S}\Vert
_{L^{p}(0,L;X)} \\ \ecart 
& \leqslant  & \dis\frac{C}{1+|\lambda |}\left( \Vert g_{I}\Vert_{L^{p}(-\ell ,0;X)}+\Vert g_{S}\Vert _{L^{p}(0,L;X)}\right)  \\ \ecart 
& \leqslant  & \dis\frac{C}{1+|\lambda |}\left( \Vert n_{I}\Vert_{L^{p}(-\ell ,0;X)}+\Vert n_{S}\Vert _{L^{p}(0,L;X)}\right)  \\ \ecart 
& \leqslant  & \dis\frac{C}{1+|\lambda |}\left( \Vert n_{I}\Vert_{L^{p}(\Omega _{I})}+\Vert n_{S}\Vert _{L^{p}(\Omega _{S})}\right)  \\ \ecart 
& \leqslant  & \dis\frac{C}{1+|\lambda |}\,\Vert n\Vert_{L^{p}(\Omega )}.
\end{array}
\end{equation*}
The same techniques are used, without assumption \eqref{hyp const}, in order to obtain the estimates concerning the juveniles and the adults, that is
\begin{equation*}
\Vert j\Vert _{L^{p}(\Omega )}=\Vert j_{I}\Vert _{L^{p}(\Omega _{I})}+\Vert
j_{S}\Vert _{L^{p}(\Omega _{S})}\leqslant \frac{C}{1+|\lambda |}\,\Vert
k\Vert _{L^{p}(\Omega )},
\end{equation*}%
and 
\begin{equation*}
\Vert a\Vert _{L^{p}(\Omega )}=\Vert a_{I}\Vert _{L^{p}(\Omega _{I})}+\Vert
a_{S}\Vert _{L^{p}(\Omega _{S})}\leqslant \frac{C}{1+|\lambda |}\,\Vert
m\Vert _{L^{p}(\Omega )}.
\end{equation*}%
Therefore, in $\E=\left( L^{p}(\Omega )\right) ^{3}$, we conclude that 
\begin{equation*}
\left\Vert \left( 
\begin{array}{c}
j \\ 
a \\ 
h
\end{array}
\right) \right\Vert _{\E}=\max \left( \Vert j\Vert _{L^{p}(\Omega )},\Vert
a\Vert _{L^{p}(\Omega )},\Vert h\Vert _{L^{p}(\Omega )}\right) \leqslant 
\frac{C}{1+|\lambda |}\left\Vert \left( 
\begin{array}{c}
k \\ 
m \\ 
n%
\end{array}%
\right) \right\Vert _{\E},
\end{equation*}
which complete the proof of Theorem \ref{Th principal}. \\

Furthermore, it would be interesting to study a system equivalent to (\ref{Syst Reaction-diffusion}) with density dependance using , for example, Ricker recruitment function for hosts. \\
Regarding the vectors, no condition is required on them. And it would be interesting to add density dependence.

\section{Conclusion}

In this work, we have modeled the sylvatic transmission of Chagas disease by incorporating vertical transmission within host populations. The primary originality of this research lies in considering, for the first time, that a healthy individual (vector or host) can move from an infected area toward a safe one via a Brownian motion, which better reflects field observations. 
The constructed system of reaction-diffusion equations was formulated as the following evolution equation:
\begin{equation}\label{Pb evolution 2}
\left\{ 
\begin{array}{lll}
V^{\prime }(t) & = & (\L + \B) V(t) \\ 
\ecart V(0) & = & V_{0},
\end{array}
\right.
\end{equation}
for a suitable initial condition $V_0$. 

From a mathematical standpoint, the originality of this work is the generation of an analytic semigroup by $\L$ (see Theorem \ref{Th principal} and the link with sectoriality in Definition \ref{Def op sect}), specifically proving that $-\L$ is sectorial with an angle strictly less than $\pi/2$. This result, along with the satisfaction of condition (\ref{hyp const}), depends solely on demographic parameters, spatial dispersion, and the vertical transmission rate of the hosts. Notably, if the host population is density-dependent, condition (\ref{hyp const}) is naturally satisfied.

Establishing these properties, particularly the invertibility of the determinant $D_\lambda$ defined in section \ref{Sect Calcul Det}, required several advanced mathematical tools:
\begin{itemize}
\item the $H^\infty$-calculus,
\item specific analytic semigroup properties,
\item the notion of BIP operators,
\item real interpolation space theory.
\end{itemize}
Using the Kato's perturbation theory, see \cite{kato1980}, Chapter 4, p. 189 and Theorem 2.4, p. 499, we can prove that $\L + \B$ generates an analytic semigroup.\\
Note that $\E=\left( L^{p}(\Omega )\right) ^{3}$ is a UMD space, see \cite{bourgain1983} and \cite{burkholder1981}, which implies that $L^p(0,T;\E)$, $p\in (1,+\infty)$ is also a UMD space.
On the other hand, we can also show the BIP character of $-(\L +\B)$ in $\E$.  \\ 
Finally, these results allow the application of the Dore-Venni sum theory \cite{dore-venni1987} to achieve the complete resolution of the evolution problem \eqref{Pb evolution 2}. It should be emphasized, however, that the numerical exploitation of model (\ref{Syst Reaction-diffusion}) remains contingent upon the subsequent proof of uniqueness and regularity of the solution.

\section*{Acknowledgments} 
The authors would like to thank very much the referee for the valuable comments and corrections which have helped us to improve this paper.

\end{document}